\documentclass[11pt]{amsart}

\pdfoutput=1

\usepackage{amsthm}
\usepackage{amsmath}
\usepackage{amssymb}

\usepackage{algorithm2e}
\makeatletter
\renewcommand{\@algocf@capt@plain}{above}
\makeatother

\usepackage[normalem]{ulem}

\usepackage{a4wide}
\usepackage{bbm}
\usepackage{graphicx}
\usepackage{tikz}
\usetikzlibrary{positioning}

\usepackage{longtable}

\usepackage{caption}
\usepackage{subcaption}

\newcommand{\RR}{\mathbbm{R}}

\theoremstyle{plain} 
\newtheorem{theorem}{Theorem}
\newtheorem{lemma}[theorem]{Lemma}
\newtheorem{proposition}[theorem]{Proposition}
\newtheorem{corollary}[theorem]{Corollary}

\theoremstyle{definition}  
\newtheorem{definition}{Definition}
\newtheorem{example}{Example}

\theoremstyle{remark}  

\newcommand{\F}{{\mathcal F}}

\newcommand{\R}{{\mathbb R}}

\newcommand{\M}{{\mathcal M}}

\newcommand{\rank}{\text{rank}}
\newcommand{\diag}{\operatorname{diag}}

\DeclareMathOperator{\EM}{EM}

\begin{document}

\title{Maximum likelihood estimation of the Latent Class Model through model boundary decomposition}



\author{Elizabeth Allman}
\email{e.allman@alaska.edu}
\address{Department of Mathematics and Statistics, University of Alaska Fairbanks}

\author{Hector Ba\~{n}os Cervantes}
\email{hdbanoscervantes@alaska.edu}
\address{Department of Mathematics and Statistics, University of Alaska Fairbanks}

\author{Robin Evans}
\email{evans@stats.ox.ac.uk}
\address{Department of Statistics, University of Oxford}

\author{Serkan Ho\c{s}ten}
\email{serkan@sfsu.edu}
\address{Department of Mathematics, San Francisco State University}

\author{Kaie Kubjas}
\email{kaie.kubjas@aalto.fi}
\address{Department of Mathematics and Systems Analysis, Aalto University;\newline
\indent Laboratory for Information \& Decision Systems, Massachusetts Institute of Technology;\newline
\indent Laboratoire d'Informatique, Sorbonne Universit\'e}

\author{Daniel Lemke}
\email{dlemke01@gmail.com}
\address{Department of Mathematics, San Francisco State University}

\author{John Rhodes}
\email{j.rhodes@alaska.edu}
\address{Department of Mathematics and Statistics, University of Alaska Fairbanks}

\author{Piotr Zwiernik}
\email{piotr.zwiernik@upf.edu}
\address{Department of Economics and Business, Universitat Pompeu Fabra, Barcelona}

\thanks{Kubjas was supported by the European Union's Horizon 2020 research and innovation programme 
(Marie Sk\l{}odowska-Curie grant agreement No 748354). Zwiernik acknowledges the support of the Beatriu de Pin\'{o}s fellowship of the Government of Catalonia's Secretariat for Universities and Research of the Ministry of Economy and Knowledge.}

\newcommand{\indep}{{\;\bot\!\!\!\!\!\!\bot\;}}
\newcommand{\tcm}{\textcolor{magenta}}

\pagestyle{plain}

\begin{abstract}
  The Expectation-Maximization (EM) algorithm is routinely used for
  maximum likelihood estimation in latent class analysis.
  However, the EM algorithm comes with no global guarantees of reaching the
  global optimum. We study the geometry of the latent class model in
  order to understand the behavior of the maximum likelihood
  estimator. In particular, we characterize the boundary
  stratification of the binary latent class model with a binary hidden
  variable.  For small models, such as for three binary observed
  variables, we show that this stratification allows exact computation
  of the maximum likelihood estimator.
  In this case we use simulations to study the maximum likelihood estimation attraction basins
  of the various strata and performance of the EM algorithm.  Our theoretical study is complemented with a
  careful analysis of the EM fixed point ideal which provides an
  alternative method of studying the boundary stratification and maximizing the likelihood function. In particular, we compute
  the minimal primes of this ideal in the case of a binary
  latent class model with a binary or ternary hidden random
  variable.
\end{abstract}


\date{28  July 2018}

\maketitle

\section{Introduction}

Latent class models are popular models used in social sciences and
machine learning. They were introduced in the 1950s by Paul
Lazarsfeld~\cite{lazarsfeld1950logical} and were used to find groups
in a population based on a hidden attribute (see
also~\cite{lazarsfeld1968lsa}). The model obtained its modern
probabilistic formulation
in the 1970s (e.g. ~\cite{Goodman1979}); we refer to~\cite{FienbergLC}
for a more detailed literature review.  More recently, latent class
models have also become widely used in machine learning, where they are
called naive Bayes models. They are a popular method of text
categorization, and are also used in other classification
schemes~\cite[Section~8.2.2]{bishop}.

The latent class model is an instance of a model with incomplete
data. Maximum likelihood estimation in such models may be
challenging, and is typically done using the EM
algorithm~\cite{EM}. Stephen Fienberg in his
discussion~\cite{EMdiscussion} of the paper introducing the EM
algorithm shed some light on its potential problems---his comments 
are relevant to the latent class
model. Referring to~\cite{haberman} Fienberg noted two main problems:
(a) even for cases where the log-likelihood for the problem with complete
data is concave, the log-likelihood for the incomplete problem need not be
concave, and (b) the likelihood equations may not have a solution
in the interior of the parameter space. He then wrote:
\begin{quote}
  In the first case multiple solutions of the likelihood equations can
  exist, not simply a ridge of solutions corresponding to a lack of
  identification of parameters, and in the second case the solutions
  of the likelihood equations occur on the boundary of the parameter
  space,[\ldots].
\end{quote}
The latent class model can be formulated as a graphical model with an
unobserved variable defined by a star graph like in Figure
\ref{fig:star}.  Given that the variable for the internal vertex was observed,
the underlying model becomes a simple instance of an exponential
family and, consequently, admits a concave log-likelihood function with a closed
formula for the maximizer.  However, the marginal likelihood will
typically have many critical points, and the maxima may lie on the
boundary of the model. In practice, to avoid the boundary problem, Bayesian
methods need to be employed to push the solutions away from the
boundary by using appropriate priors \cite{garre2006avoiding}.

\begin{figure}[htp]
\tikzstyle{vertex}=[circle,fill=black,minimum size=5pt,inner sep=0pt]
\tikzstyle{hidden}=[circle,draw,minimum size=5pt,inner sep=0pt]
  \begin{tikzpicture}
  \node[vertex] (1) at (-.7,.7)  [label=above:$1$] {};
    \node[vertex] (2) at (-.8,-.4)  [label=left:$5$] {};
    \node[vertex] (3) at (.4,.9) [label=above:$2$]{};
    \node[vertex] (4) at (.95,.1) [label=right:$3$]{};
    \node[vertex] (5) at (.3,-0.9) [label=right:$4$]{};
    \node[hidden] (a) at (0,0) {};
    \draw[line width=.3mm] (a) to (1);
    \draw[line width=.3mm] (a) to (2);
    \draw[line width=.3mm] (a) to (3);
        \draw[line width=.3mm] (a) to (4);
    \draw[line width=.3mm] (a) to (5);
  \end{tikzpicture}
  \caption{The star graph model with $5$ leaves. The internal vertex represents an unobserved random variable.}\label{fig:star}
\end{figure}
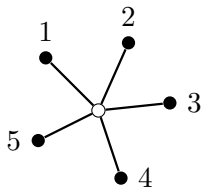

\subsection{Outline of Results} Our aim is to study the boundary problem for the latent
class model from the perspective of maximum likelihood estimation.  We
will use the link between latent class models and nonnegative
tensor rank. For instance, the latent class model with three binary
observed variables and one binary hidden variable is the model of
normalized nonnegative $2 \times 2 \times 2$ tensors of nonnegative
rank ($\rank_+$) at most two.  We will rely on recent work in
algebraic statistics on the description of the (algebraic) boundary of
tensors of nonnegative rank two~\cite{ARSZ}. Our
Theorem~\ref{thm-main1} gives a complete characterization of the
boundary strata of binary latent class models with a binary hidden
variable.

The geometry of these models allows us to identify boundary strata for
which the maximum likelihood problem is easy, such as certain
codimension two strata; see Section~\ref{subsec:general}. In
Section~\ref{subsec:MLE-M3}, we showcase the use of
Theorem~\ref{thm-main1} for the maximum likelihood estimation in the
$2 \times 2 \times 2$ case of $\rank_+ \leq 2$ by solving the problem
exactly: we provide a formula for the maximizer of the likelihood
function over the algebraic set defining each boundary strata.
Together with recent work in \cite{SM17}, this is the first
non-trivial example of the exact solution provided for a latent class
model, which typically is fitted using the EM algorithm.  The geometry
used for this exact solution also provides insight into the maximum
likelihood estimation in this model class, validating some of the
concerns of Fienberg. We report the results of our simulations which
show that the overwhelming majority of data has a maximum likelihood
estimator on the boundary of the model (where some model parameters
are zero). Indeed, under certain scenarios, even if the true underlying distribution lies in
the interior of the model, the maximum likelihood estimator may be found on the 
boundary with high probability. 
We also 
examine briefly the model of $3 \times 3 \times 2$ tensors of nonnegative rank $3$. 
Our simulations indicate that this model occupies a tiny portion (approximately $.019 \%$)
of the probability simplex.

In Section~\ref{sec:EM-fixed}, we study the algebraic description of
the fixed points of the EM algorithm inspired by~\cite{KRS}.  In
particular, we compute the irreducible components of the EM
fixed point ideal for the $2 \times 2 \times 2$ tensors of $\rank_+ \leq 2$ 
and of $\rank_+ \leq 3$.
In the first case, we demonstrate that we
can recover the formulas in Section~\ref{subsec:MLE-M3} from certain
components of the EM fixed point ideal via elimination. In the second case, the irreducible decomposition 
we compute validates the results in \cite{SM17} on the boundary decomposition of this model. 

\section{Definitions and Background}\label{sec:basics}

If $X$ is a random variable with values in $\{1,\ldots,k\}$, then its
distribution is a point $(p_1,\ldots,p_k)$ in the \emph{probability simplex}
$$\Delta_{k-1} \;\;=\;\; \{(x_1, \ldots, x_k) \in \RR^k \, : \, x_1 + \ldots + x_k =1, \,\,\, x_1,\ldots, x_k \geq 0\}.$$ 
The vector $X=(X_1,\ldots,X_n)$ is a \emph{binary random vector} if 
$X_i\in \{1,2\}$ for each $1\leq i\leq n$. 
A \emph{binary tensor} $P=(p_{i_1\cdots i_n})$, where $i_j\in \{1,2\}$, is 
a $2\times 2\times \cdots\times 2$ table 
of real numbers in $\R^{2\times \cdots\times 2}=\R^{2^n}$. A tensor is \emph{nonnegative} if it has only 
nonnegative entries. Every probability distribution for a binary vector $X=(X_1,\ldots,X_n)$ is a nonnegative 
binary tensor in the probability simplex $\Delta_{2^n-1}$: 
$$
p_{i_1i_2\cdots i_n}\;=\;{\rm Prob}(\{X_1=i_1,X_2=i_2,\ldots,X_n=i_n\}).
$$

The \emph{binary latent class model} $\M_{n,r}$ is a statistical model
for a vector of $n$ binary random variables $X=(X_1,\ldots,X_n)$. It
contains all distributions such that $X_1,X_2,\ldots,X_n$ are
independent given an unobserved random variable with $r\geq 1$
states. The model is parameterized by the distribution
$\lambda=(\lambda_1,\lambda_2,\ldots, \lambda_r)\in \Delta_{r-1}$ of
the unobserved variable and the conditional distributions of each
$X_i$ given the unobserved variable, which we write in form of a
stochastic matrix
$$ 
A^{(i)} \, = \, \left( \begin{array}{cc} a_{11}^{(i)} &  a_{12}^{(i)} \\[.3cm] \vdots & \vdots \\[.3cm] a_{r1}^{(i)} &  a_{r2}^{(i)} \end{array} \right), \;\;\;\; i=1,\ldots, n,
$$  
where
$a_{kl}^{(i)} \geq 0$  
and $a_{j1}^{(i)} + a_{j2}^{(i)}=1$ for each
$j=1,\ldots,r$. 
Letting $C_n$ denote the $n$-dimensional cube
$\Delta_1^n$, then the parameter space of 
{$\M_{n,r}$ is $\Theta:= \Delta_{r-1} \times (C_n)^r$ }
with elements 
$$
{
\theta\;\;=\;\;(
\lambda_1, \ldots, \lambda_{r}, \; \;
a_{11}^{(1)}, a_{12}^{(1)},  \ldots, a_{r1}^{(1)}, a_{r2}^{(1)}, \; \;
\ldots , \; \;
a_{11}^{(n)}, a_{12}^{(n)},  \ldots, a_{r1}^{(n)}, a_{r2}^{(n)}
).
}
$$ 
To be succinct, a choice of parameters $\theta$ is also denoted by 
$\theta = (\boldsymbol \lambda, A^{(1)}, A^{(2)}, \ldots, A^{(n)})$.
The parameterization $ \phi_{n,r}: \:  \Theta \to \Delta_{2^n-1} $  of $\M_{n,r}$ is given by
\begin{equation}\label{eq:paramMn}
\phi_{n,r} \, : \;\;\;\;\, \theta \;\;\;\;\;\mapsto\;\;\;\;\;
p_{j_1j_2\cdots j_n}(\theta) \;\;=\;\; \lambda_1   a_{1j_1}^{(1)} a_{1j_2}^{(2)} \cdots a_{1j_n}^{(n)}  + \cdots+\lambda_r   a_{rj_1}^{(1)} a_{rj_2}^{(2)} \cdots a_{rj_n}^{(n)}.	
\end{equation}

This parameterization shows that the distributions in $\M_{n,r}$
admit a decomposition into $r$ summands, which can be phrased in terms
of tensor decompositions. A binary tensor $P$ has \emph{rank one} if
it is an outer product of $n$ vectors in $\R^2$; that is, there exist
$u_1,\ldots,u_n\in \R^2$ such that $p_{i_1i_2\cdots
  i_n}=u_{1i_1}u_{2i_2}\cdots u_{ni_n}$. 
A tensor has \emph{nonnegative rank ($\rank_+$)
  at most} $r$ if it can be written as a sum of $r$ nonnegative
tensors of rank one. Equivalently, a binary tensor with $\rank_+ \leq
r$ is a point in the image of a map $\psi_{n,r}:\: (\R^{2}_{\geq
  0})^{nr}\longrightarrow \R^{2^n}$ defined as
$$\psi_{n,r} \, :\;\; \, \prod_{i=1}^n\prod_{j=1}^r (u_{j1}^{(i)}, u_{j2}^{(i)})   \;\;\;\mapsto\;\;\;
p_{i_1i_2\cdots i_n} \;=\; \sum_{j=1}^r u_{ji_1}^{(1)} u_{ji_2}^{(2)} \cdots u_{ji_n}^{(n)}  .$$

For more on the connection between tensor rank, nonnegative tensor rank,
and several of the latent class models under consideration here,
see, for example, \cite{ARM09, dSL} or for a connection to phylogenetic models \cite{ART2014}.
Here we simply formulate the following result.
\begin{proposition}
The set of binary tensors with $\rank_+ \leq r$ is  the cone over the binary latent class model $\M_{n,r}$. 	
\end{proposition}

In this paper we focus primarily on models with two latent classes,
$r=2$, and write $\M_n:=\M_{n,2}$ and $\phi_n:=\phi_{n,2}$. 
This case in some ways is
`easy' since both the algebraic boundary (Proposition \ref{prop-boundary}) and the singular set of the parameterization 
map (Proposition \ref{prop-sing}) are well understood.  When considering questions of
higher nonnegative rank, we have no such tools at our disposal.
Binary
tensors of $\rank_+ \leq 2$
were studied in~\cite{ARSZ}, where the following theorem
gives a description of $\mathcal M_n$ as a semi-algebraic set.

\smallskip

\begin{theorem}~\cite[cf. Theorem~1.1]{ARSZ}\label{th:ARSZ} A binary
  tensor $P = (p_{i_1i_2\cdots i_n})$ has nonnegative rank at most two
  if and only if $P$ has flattening rank at most two and $P$ is
  supermodular.
\end{theorem}

A matrix flattening of the binary tensor $P$ is a $2^{|\Gamma|} \times 2^{n -
  |\Gamma|}$ matrix where $\Gamma \subset \{1, \ldots, n\}$ with $1
\leq |\Gamma| \leq n-1$. 
The flattening rank is the maximal rank of any of
these matrices. This rank condition provides the equations for the
semi-algebraic description since the rank of a matrix is at most two
if and only if all $3$-minors of that matrix vanish. 
Now, we briefly
also explain the supermodularity.  Let $\pi = (\pi_1, \ldots, \pi_n)$
be an $n$-tuple of permutations $\pi_j \in S_2$. We say $P$ is
$\pi$-supermodular if
\begin{equation}\label{eq:logsuper}
 p_{i_1i_2\cdots i_n} \, p_{j_1j_2\cdots j_n} \;\; \leq\;\; p_{k_1k_2\cdots k_n} \, p_{\ell_1 \ell_2\cdots \ell_n} 	
\end{equation}
holds when $\{i_s, j_s\} = \{k_s, \ell_s\}$ and $\pi_s(k_s) \leq \pi_s(\ell_s)$ for $s=1, \ldots, n$. The tensor $P$ is supermodular if it is $\pi$-supermodular
for some $\pi$.

\begin{corollary}
	The semi-algebraic description of the binary latent class model is given by Theorem~\ref{th:ARSZ} together with the extra constraint that 
$\sum_{i_1,\ldots, i_n} p_{i_1\cdots i_n}=1$.
\end{corollary}

We close this section with a result that simplifies some arguments regarding the boundary stratification of $\M_n$. 
\begin{lemma}\label{lem:marg}
  Let $\M_{n,r}$ be the latent class model for the random vector
  $X=(X_1,\ldots,X_n)$, and let $B\subset \{1,\ldots,n\}$ with $|B| =
  m$. Then the induced marginal model for $X_B = (X_i\, : \, i \in B)$
  is $\M_{m,r}$. In particular, if $P$ is a tensor in $\M_{n,r}$ given
  by parameters $(\lambda_1,\ldots,\lambda_r$, $A^{(1)},\ldots,
  A^{(n)})$, then the corresponding marginal distribution $P_B$ is
  given by parameters $\lambda_1,\ldots,\lambda_r$, and $A^{(i)}$ for
  $i\in B$.
\end{lemma}\begin{proof}
  The marginal distribution $P_B$ is obtained from $P = (p_{j_1\ldots
    j_n})$ by summing over all indices $j_k$ with $k\notin B$.  When
  we compute the sum using the parameterization (\ref{eq:paramMn}) the
  result follows because $a_{i1}^{(k)}+ a_{i2}^{(k)}=1$ for all
  $i=1,\ldots,r$.
\end{proof}

\section{Boundary Stratification of $\M_{n}$}\label{sec:stratification}


The semi-algebraic description of  $\M_{n}$ can also be used to understand the topological boundary of this set. 
When $n= 1, 2$, $\M_n$ is well-understood:  $\M_1=\Delta_1$ and 
$\M_2=\Delta_3$ respectively; see, e.g., \cite[Corollary 2.2]{gilula1979singular}.
Thus we focus on the case of three or more observed variables, and assume that $n\geq 3$ throughout. 
We begin our analysis with the following proposition.
\begin{proposition} \label{prop-boundary} The dimension of the model $\M_{n}$ is $2n+1$. The boundary of this 
semi-algebraic set is defined by $2n$ irreducible components. 
Each component is
 the image of the set in the domain of $\phi_n$ given by $a_{1j}^{(i)} = 0$ for $i=1, \ldots, n$ and $j=1,2$.
\end{proposition}
\begin{proof} The dimension of $\M_{n}$ is the number of independent parameters in the domain of $\phi_n$. 
This follows because this model is generically identifiable, which is classically well known; see, e.g \cite{mchugh1956efficient}. 
The statement about 
the boundary is Theorem~1.2 in~\cite{ARSZ}, and the statement about each component is found in the proof of the same result. 
\end{proof}
Observe that these components are also defined by $a_{2j}^{(i)}=0$,
but one can interchange the rows of the matrices
$A^{(i)}$ 
and the entries of $\lambda$,
and get the same points on the boundary. 
This corresponds to `label swapping' on the latent variable.
Each component
is the collection of tensors where one slice has rank one. By a 
\emph{slice} of a tensor $P=(p_{i_1\cdots i_n})$, we mean a subtensor 
obtained by fixing one index $i_k$.
We note that for general tensors with nonnegative rank bigger
than two, the boundary of the corresponding model $\M$ is not well
understood. For instance, points on the boundary of the parameter
space defined by setting one parameter equal to zero no longer map to
the boundary of the model $\M$; see Example 5.2 in~\cite{ARSZ}.
A recent development is \cite{SM17} where the boundary of $\M_{3,3}$ has
been described.

In this paper, we consider also lower dimensional pieces of the
boundary of $\M_{n}$. Our motivation is to perform maximum likelihood
estimation over such models
efficiently. Proposition~\ref{prop-boundary} implies that various
intersections of the $2n$ irreducible codimension one components
define lower dimensional boundary pieces. We call a set of boundary
points of dimension $k$ obtained as such an intersection~a
\emph{$k$-dimensional stratum}.
We will identify and describe the boundary strata that are relevant
for maximum likelihood estimation. The relevant boundary strata are
those which are not
degenerate. 
\begin{definition}
  The degenerate part of $\Delta_{2^n-1}$ is the set of tensors $P =
  (p_{i_1 \cdots i_n})$ where for fixed $1 \leq j < k \leq n$ and a
  choice $i_j = s$ and $i_k = t$ with $s,t \in \{1,2\}$ the entries
  $p_{i_1 \cdots s \cdots t \cdots i_n} = 0$ for all $i_u$, $u \neq
  j,k$. \end{definition}
Another way of detecting that a binary tensor $P$ is degenerate is to
look at the marginal table $P_{\{j,k\}}$. If any of the entries of
this $2 \times 2$ table is zero for any $j,k$, then $P$ is degenerate. For instance,
if
$$P_{\{j,k\}} = \left( \begin{array}{cc} p_{11} & p_{12} \\ p_{21} & p_{22} \end{array} \right) \, = \, \left( \begin{array}{cc} \alpha & 0 \\ \beta & \gamma \end{array} \right)$$
with $\alpha, \beta, \gamma > 0$, then knowing that $X_j = 1$ implies
that $X_k =1$. By restricting to nondegenerate tensors, we avoid
this kind of probabilistically degenerate situation. We could have formulated our main theorem only 
for the interior of the probability simplex, since from a mathematical point of view, extending it to some 
parts of the boundary seems like an incremental gain.   From the statistical point of view, however,
this gain is quite  dramatic as it allows us to 
understand the maximum likelihood estimator even when data tables contain zeros 
(as long as two-way marginal tables have no zeros).
This is especially important for validating the simulations in 
Section \ref{sec:MLE}, when the sample sizes are relatively small.

\subsection{Singular locus of the parametrization map}

To state and prove our main result, we need an understanding of the singular
locus of the parameterization map $\phi_n$.   Recall that a
point of the domain $\Theta$ is a singular point of $\phi_n$ if the
Jacobian of the map drops rank at this point.  To describe this
set, we look at various (overlapping) subsets of the parameter space.
Specifically, let
 
\smallskip

\begin{enumerate}
\item[] $\Theta_{\lambda_1\lambda_2} \subseteq \Theta$ be the subset defined by $\lambda_1\lambda_2=0$; 
\item[] \hskip .4cm $\Theta_{ij} \subseteq \Theta$ be the subset where $\rank(A^{(k)}) = 1$ for all $k \neq i,j$ with $1 \leq i \neq j \leq n$;  
and
\item[] \hskip .5cm $\Theta_j \subseteq \Theta$ be the subset where $\rank(A^{(k)}) = 1$ for all $k \neq j$.
\end{enumerate}

\smallskip

\noindent
Finally, we denote by $\Theta^1$ the subset of $\Theta$
where $\rank(A^{(k)}) = 1$ for all $k$.   It is clear that $\Theta_{ij} = \Theta_{ji}$
and $\Theta^1 \subset \Theta_k \subset \Theta_{jk}$ for all $1 \leq j \neq k \leq n$. 

\smallskip

The probabilistic interpretation of these special loci is simple. The set $\Theta_{\lambda_1\lambda_2}$ corresponds to the parameters for which the latent variable is 
degenerate taking always the value $0$, or always the value $1$. The set $\Theta_{ij}$ corresponds to the special situation where all 
variables $X_k$ for $k\neq i,j$ are probabilistically independent of the latent variable. That is, only two observable variables in the system carry 
some information about the latent one. In the case of $\Theta_j$, only $X_j$ is allowed to nontrivially depend on the latent variable, and $\Theta^1$ 
corresponds to points where all observed random variables are independent of the latent one. 
Note that the points in the sets $\Theta_{\lambda_1\lambda_2}$, $\Theta_j$, and $\Theta^1$ correspond to the situation where all observed variables are independent
of each other.  

\smallskip 
\noindent
\emph{Remark}:
For those familiar with tensor decompositions, these subsets of parameters have
simple descriptions in terms of the ranks of the matrices $A^{(i)}$.  
Suppose that 
$\theta = (\boldsymbol \lambda, A^{(1)}, \ldots, A^{(n)})$, then the
\emph{$m$-rank} of $\theta$ is the $n$-tuple 
$(\rank(A^{(1)}), \ldots,  \rank(A^{(n)}) )$.  In this setting,
we see that, for example,  $\Theta_{12}$ corresponds to parameters
with $m$-rank $(r_1, r_2, 1, 1, \ldots, 1)$ with 
$r_1, r_2 \le 2$.   The subset $\Theta_1$ corresponds to parameters
with $m$-rank $(r_1,1,1, \ldots, 1)$, and $\Theta^1$ to those
parameters with $m$-rank $(1,1,\dots,1)$.  Indeed,
this perspective makes it quite easy to determine both the singular locus of $\phi_n$
and the tensor rank of the images of these parameter sets.

\smallskip 
\begin{proposition}\label{prop-sing} 
  The singular locus of the parametrization map $\phi_n$ is equal to
  $$\Theta_{\lambda_1\lambda_2} \cup \bigcup_{1 \leq i \neq j \leq n}
  \Theta_{ij}.$$
\end{proposition}
This result is not new, cf. \cite[Corollary 7.17]{MOZ} and could also be inferred from Theorems 13 and 
14 in \cite{geiger2001stratified}. We provide an alternative proof that is based on ideas from 
\cite{ARM09} and \cite{dSL}. 
\begin{proof}[Proof of Proposition \ref{prop-sing}]
It is clear that the sets $\Theta_{\lambda_1 \lambda_2}$ and $\Theta^1$ map
under $\phi_n$ to distributions in $\mathcal M_n$  of nonnegative rank $1$,  and thus that
the Jacobian drops rank at these points.  A simple computation shows 
that $\phi_n$ maps points in $\Theta_k$ to tensors of $\rank_+ = 1$, and the Jacobian is rank deficient at these parameter points too.  
Consider now those parameters $\theta$ with (up to permutation) $m$-rank 
$(2,2,1,\dots,1)$ and, without loss of generality, $\theta \notin \Theta_{\lambda_1, \lambda_2}$.
Let $P_\theta = \phi_n ( \theta )$.  We quickly show that $P_\theta$ has nonnegative rank
2, and that $\theta$ is a singular point of the parameterization.  Since $A^{(3)}, \ldots, A^{(n)}$
are singular matrices, let $\mathbf v$ be the tensor product of their top rows.  Stated
in more statistical language, $\mathbf v$ is the (vectorized) joint distribution of the independent 
binary variables $X_3, \ldots, X_n$. Using $A^{(1)}, A^{(2)}$ for the matrix parameters of rank $2$,
then the joint distribution $P_\theta$ is
$P_\theta = (A^{(1)})^T \, \diag ( [\lambda_1, \lambda_2 ]) \, A^{(2)} \otimes \mathbf v$.
Since $(A^{(1)})^T \, \diag ( [\lambda_1, \lambda_2 ]) \, A^{(2)} $ is a rank 2 matrix, $P_\theta$
is a rank 2 tensor.  However, 
the fiber of $P_\theta$ is positive dimensional.  This follows because the matrix
factorization $(A^{(1)})^T \, \diag ( [\lambda_1, \lambda_2 ]) \, A^{(2)}$ above is not unique.
If $\Sigma$ is taken to be any matrix sufficiently close to the identity and with column sums equal to $1$, 
then $\tilde{A}^{(1)} = \Sigma^T A^{(1)}$ is Markov, 
$\tilde{\boldsymbol \lambda} = \Sigma^{-1}  \diag ( [\lambda_1, \lambda_2 ]) A^{(2)}  \begin{pmatrix}
1 \\ 1 \end{pmatrix}$ has positive
entries, $\tilde{A}^{(2)} = \diag(\tilde{\boldsymbol \lambda})^{-1} \Sigma^{-1} \diag ( [\lambda_1, \lambda_2 ]) A^{(2)}$
is Markov, and $\phi_n ( \tilde{\boldsymbol \lambda}, \tilde{A}^{(1)}, \tilde{A}^{(2)}, A^{(3)}, \ldots A^{(n)}) $ also equals 
$P_\theta$.  It follows that $\theta$ is a singular point of the parameterization $\phi_n$.

Finally, consider parameters $\theta$ of $m$-rank $(2,2,2,r_4, \ldots, r_n)$ up to permutation,
$\theta \notin \Theta_{\lambda_1 \lambda_2}$.  Then by Kruskal's
Theorem \cite{Kruskal76, Kruskal77} together with techniques developed in \cite{ARM09}
for proving parameter identifiability, $\theta$ is identifiable and the fiber of $P_\theta$ is of size $2$.  This means that $\theta$ is 
not a singular point of $\phi_n$.
\end{proof}

We now state and prove two lemmas used repeatedly 
in the proof of Theorem \ref{thm-main1}.

\begin{lemma} \label{lem-singdim}
$\phi_n(\Theta_{ij})$ is an $(n+1)$-dimensional subset of $\Delta_{2^n-1}$ isomorphic to $\Delta_{3}\times(\Delta_1)^{n-2}$.
\end{lemma}
\begin{proof}
	A $2\times 2$ stochastic matrix has rank one if and only if both of its rows are equal. 
Therefore, points in the image of $\Theta_{ij}$ are of the form 
$$p_{k_1\cdots k_n}\;=\;(\lambda_1a_{1k_i}^{(i)}a_{1k_j}^{(j)}+\lambda_2a_{2k_i}^{(i)}a_{2k_j}^{(j)})\prod_{l\neq i,j}a_{1k_l}^{(l)}.$$ 
It is clear that $\phi_n(\Theta_{ij})$ is a subset of $\Delta_{2^n-1}$ 
isomorphic to $\M_2 \times(\Delta_1)^{n-2}$. The equality follows because $\M_2=\Delta_3$.
\end{proof}

\begin{lemma} \label{lem-singdegenerate} The parametrization $\phi_n$
  maps $\Theta_{ij} \cap \{a_{st}^{(k)}=0\}$ for $k \neq i,j$ and
  $s,t\in\{1,2\}$ to the degenerate part of the boundary of
  $\Delta_{2^n-1}$.
\end{lemma}
\begin{proof} Consider the case $a_{11}^{(k)}=0$. Then
  $a_{12}^{(k)}=1$, and since $A^{(k)}$ has rank one we conclude that
  $a_{21}^{(k)} =0$ and $a_{22}^{(k)} = 1$. This means that the first
  slice of the image tensor along dimension $k$ is identically zero.
  Similar reasoning applies for all $a_{st}^{(k)} = 0$.
\end{proof}
Below we consider the intersection of
various subsets of the boundary of $\Theta$ with pieces of the
singular locus. Motivated by the last lemma, we denote
$\Theta_{ij} \cap \mathrm{int}(\Theta)$ by $\Theta_{ij}^\circ$. We
also let $\Theta_j^\circ = \Theta_j \cap \mathrm{int}(\Theta)$.

\subsection{Main Theorem}

We now state our main theorem.
\begin{theorem} \label{thm-main1} For $ n \leq k \leq 2n+1$, the
  $k$-dimensional strata of the nondegenerate part of $\M_{n}$ are in
  bijection with the $k-(n+1)$-dimensional faces of the cube $C_n$,
  except for $k=2n-1$ when $n$ additional strata are present, and for $k= n+1$
  when ${n \choose 2}$ additional strata are present. More precisely, the stratification of $\M_n$ has five types of strata: 
\begin{enumerate}

\item The interior of $\M_{n}$. This strata has dimension $2n+1$ and each point is
  the image under $\phi_n$ of a nonsingular point in the interior of
  $\Theta$.

\item Non-exceptional strata of dimension $n+1\leq k\leq 2n$. Except for $k=2n-1$, each $k$-dimensional 
stratum is the image of points in 
$$   \left( \bigcap_{s_i: i \in I} \{a_{1s_i}^{(i)}=0\} \right) \,\, \bigcup \,\, \left( \bigcap_{s_i: i \in I} \{a_{2s_i}^{(i)}=0\} \right),$$ 
where $|I| = 2n+1-k$. For $k=2n-1$, a stratum corresponding to a codimension two face of
$C_n$ is the image of points in
$$  \{a_{1s}^{(i)}=0\} \cap \{a_{1 t}^{(j)}=0\}  \, \bigcup \, \{a_{2s}^{(i)}=0\} \cap \{a_{2 t}^{(j)}=0\} \, \bigcup \, \Theta_{ij}^\circ ,$$
for $1 \leq i < j \leq n $ and $s,t=1,2$.

\item Exceptional strata of dimension $2n-1$. These are $n$ additional strata given as the image of points in
$$ \{a_{11}^{(i)} =0\} \cap \{a_{22}^{(i)} =0\} \bigcup  \{a_{12}^{(i)} =0\} \cap \{a_{21}^{(i)} =0\},$$
for $i=1,\ldots, n$.

\item Exceptional strata of dimension $n+1$. These are ${n \choose 2}$ additional strata given as the image of points
  in $\Theta_{ij}^\circ$ for $1 \leq i < j \leq n$.
  
\item A single $n$-dimensional stratum corresponding to the
  empty face of $C_n$ given by the image of points in
  $\Theta_{\lambda_1\lambda_2}$.
\end{enumerate}
\end{theorem}
\begin{corollary} \label{cor:number_of_strata} Let $n \leq k \leq
  2n+1$ with $k = 2n+1 - \ell$. Then the number of nondegenerate
  $k$-dimensional strata of $\M_{n}$ is
$$ \left\{ 
\begin{array}{ll} \displaystyle {n \choose \ell} 2^\ell  &\qquad  \ell \neq 2,n, n+1 \\[12pt] 
\displaystyle {n \choose 2}4 + n &\qquad    \ell = 2\\[12pt] 
\displaystyle 2^n + {n \choose 2} & \qquad \ell = n \\[12pt] 
1 & \qquad  \ell = n+1. 
\end{array} \right.$$ 
\end{corollary}

We prove Theorem \ref{thm-main1} at the end of
this section, after making a few comments about the stratification.  
As a general rule, the set of probability distributions contained in a single stratum does not
allow a clean and simple interpretation. 
In a few cases, however, we \emph{do} observe
nice patterns, and we describe these below. \\

\vskip 0.1cm
(a) \emph{Codimension one strata}. The $2n$ codimension one
strata have a simple recursive description. For example, if
$a^{(1)}_{11} = 0$ then $a^{(1)}_{12}=1$, and the slice
$(p_{1j_2\cdots j_n})$ of the tensor $P$ is a binary tensor of rank
one. This corresponds to the context specific independence model where
$X_2, \ldots, X_n$ are independent conditionally on $\{X_1 = 1\}$. It
is described in the probability simplex $\Delta_{2^{n-1}-1}$ by the
binomial equations
\begin{equation}\label{eq:Gamma}
  p_{1i_2\cdots i_n}p_{1j_2\cdots j_n}-p_{1k_2\cdots k_n}p_{1l_2\cdots l_n}\;=\;0\quad \mbox{for  }\{i_s,j_s\}=\{k_s,l_s\} \mbox{ and } s=2,\ldots,n.	
\end{equation}
The other slice $(p_{2j_2 \cdots j_n})$, after normalization, is a
tensor from the model $\M_{n-1}$. Hence, knowing the description of
$\M_{n-1}$ helps describe the codimension one strata of $\M_n$. \\

\vskip 0.1cm
(b) \emph{The exceptional codimension two strata (type
  (3))}.  If $A^{(1)}$ is the identity matrix, then the
parameterization in (\ref{eq:paramMn}) specializes to
$$
p_{1j_2\cdots j_n} = \lambda_1 a_{1j_2}^{(2)}\cdots a_{1j_n}^{(n)},\qquad p_{2j_2\cdots j_n} = \lambda_2 a_{2j_2}^{(2)}\cdots a_{2j_n}^{(n)}.
$$
Since $A^{(2)},\ldots, A^{(n)}$ are arbitrary stochastic matrices,
the first stratum of type (3) corresponds to the 
model where $X_2,\ldots,X_n$ are independent conditionally on
$X_1$. This is a graphical model given by the graph in Figure
\ref{fig:star2}. This model is fully described in the probability
simplex $\Delta_{2^n-1}$ by the binomial equations
$$
p_{ii_2\cdots i_n}p_{ij_2\cdots j_n}-p_{ik_2\cdots k_n}p_{il_2\cdots
  l_n}\;=\;0\quad \mbox{for $i=1,2$, }\{i_s,j_s\}=\{k_s,l_s\}
\mbox{ and } s=2,\ldots,n$$
with no additional inequalities. The analysis is analogous for the $n-1$ remaining strata given by one of $A^{(2)},\ldots, A^{(n)}$ being the identity matrix. \\

\vskip 0.1cm
(c) \emph{The $n$-dimensional stratum (type(5))}. This
unique stratum is given by all rank one tensors in
$\Delta_{2^n-1}$. This stratum is defined by the equations
$$
p_{i_1i_2\cdots i_n}p_{j_1j_2\cdots j_n}-p_{k_1k_2\cdots k_n}p_{l_1l_2\cdots l_n}\;=\;0\quad \mbox{if }\{i_s,j_s\}=\{k_s,l_s\} \mbox{ for } s=1,\ldots,n$$ 
and it corresponds to the full independence model.

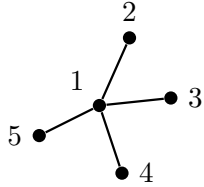
\begin{figure}[htp]
\tikzstyle{vertex}=[circle,fill=black,minimum size=5pt,inner sep=0pt]
\tikzstyle{hidden}=[circle,draw,minimum size=5pt,inner sep=0pt]
  \begin{tikzpicture}
    \node[vertex] (2) at (-.8,-.4)  [label=left:$5$] {};
    \node[vertex] (3) at (.4,.9) [label=above:$2$]{};
    \node[vertex] (4) at (.95,.1) [label=right:$3$]{};
    \node[vertex] (5) at (.3,-0.9) [label=right:$4$]{};
    \node[vertex] (a) at (0,0) [label=above left:$1$]{};
    \draw[line width=.3mm] (a) to (2);
    \draw[line width=.3mm] (a) to (3);
        \draw[line width=.3mm] (a) to (4);
    \draw[line width=.3mm] (a) to (5);
  \end{tikzpicture}  \caption{The graph representing the strata given by $a^{(1)}_{12}=a^{(1)}_{21}=0$. }\label{fig:star2}
\end{figure}

The strata of $\M_n$ form a partially ordered set where for two strata
$S,S'$ we have $S\preceq S'$ if the closure of $S$ is contained in the
closure of $S'$.  Such a partially ordered set structure becomes important in Section \ref{sec:MLE} to provide further insights into the geometry of the maximum likelihood estimation. Suppose that $p^*$ is a maximizer of a function $f$ over the (Zariski) closure of a set $S'$. If $S$ is another set such that $S\preceq S'$ then the value of $f$ in $S$ is bounded above by $f(p^*)$. In particular, if $p^*$ lies in $S'$ then to maximize $f$ over $\M_n$ there is no need to check strata $S$ such that  $S\preceq S'$.
 
The interior of $\M_n$ is the unique
maximal element, and the unique strata of type (5) is the unique
minimal element.
For example, for $\M_3$ there are six dimension 6 strata, which we label by $\{1,2,3,4,5,6\}$ corresponding to equations \\
$$(1)\;\;\;p_{111} p_{122} = p_{112}p_{121}\qquad (2)\;\;\;p_{211} p_{222} = p_{212} p_{221}$$ 
$$(3)\;\;\;p_{111} p_{212} = p_{112} p_{211}\qquad (4)\;\;\; p_{121} p_{222} = p_{122} p_{221} $$ 
$$(5)\;\;\;p_{111} p_{221} = p_{121} p_{211}\qquad (6)\;\;\; p_{112} p_{222} = p_{122} p_{212}$$ 
respectively. These six equations are naturally grouped in pairs as indicated by the three rows above. 
Each of these three pairs defines one of the three special strata of type (3). 
In general, each special stratum of this kind is obtained as the intersection~of codimension one strata which correspond 
to ``opposite'' facets of $C_n$. For $\M_3$, each special stratum of type (4) is  defined by four equations found in two rows of the six equations above. 
If one ignores these special strata, the poset is isomorphic 
to the face poset of the cube $C_n$. 
The Hasse diagram of the poset for $\M_3$ is given in Figure \ref{fig:poset}. 
\begin{figure}[htp!]
\begin{tikzpicture}[scale=.7]
  \node (max) at (0,11) {$\emptyset$};
  \node (1) at (-5,10) {$1$};
  \node (2) at (-3,10) {$2$};
  \node (3) at (-1,10) {$3$};
  \node (4) at (1,10) {$4$};
  \node (5) at (3,10) {$5$};
  \node (6) at (5,10) {$6$};
  \node (12) at (-10.5,8) {$\textcolor{red}{12}$};
  \node (13) at (-9,8) {$13$};
  \node (14) at (-7.5,8) {$14$};
  \node (15) at (-6,8) {$15$};
  \node (16) at (-4.5,8) {$16$};
  \node (23) at (-3,8) {$23$};
  \node (24) at (-1.5,8) {$24$};
  \node (34) at (0,8) {$\textcolor{red}{34}$};
  \node (25) at (1.5,8) {$25$};
  \node (26) at (3,8) {$26$};
  \node (35) at (4.5,8) {$35$};
  \node (36) at (6,8) {$36$};
  \node (45) at (7.5,8) {$45$};
  \node (46) at (9,8) {$46$};
  \node (56) at (10.5,8) {$\textcolor{red}{56}$};
  \node (1234) at (-10,3) {$\textcolor{blue}{1234}$};
  \node (135) at (-8,3) {$135$};
  \node (136) at (-6,3) {$136$};
  \node (145) at (-4,3) {$145$};
  \node (146) at (-2,3) {$146$};
  \node (1256) at (0,3) {$\textcolor{blue}{1256}$};
  \node (235) at (2,3) {$235$};
  \node (236) at (4,3) {$236$};
  \node (245) at (6,3) {$245$};
  \node (246) at (8,3) {$246$};
  \node (3456) at (10,3) {$\textcolor{blue}{3456}$};
  \node (min) at (0,1) {$123456$};
  \draw (max) -- (1);
  \draw (max) -- (2);
  \draw (max) -- (3);
  \draw (max) -- (4);
  \draw (max) -- (5);
  \draw (max) -- (6);
  \draw (1) -- (12) -- (2);
  \draw (3) -- (34) -- (4);
  \draw (5) -- (56) -- (6);
  \draw (1) -- (13) -- (3);
  \draw (1) -- (14) -- (4);
  \draw (1) -- (15) -- (5);
  \draw (1) -- (16) -- (6);
  \draw (2) -- (23) -- (3);
  \draw (2) -- (24) -- (4);
  \draw (2) -- (25) -- (5);
  \draw (2) -- (26) -- (6);
  \draw (3) -- (35) -- (5);
  \draw (3) -- (36) -- (6);
  \draw (4) -- (45) -- (5);
  \draw (4) -- (46) -- (6);
  \draw (13) -- (135) -- (15) -- (135) -- (35);
  \draw (13) -- (136) -- (16) -- (136) -- (36);
  \draw (14) -- (145) -- (15) -- (145) -- (45);
  \draw (14) -- (146) -- (16) -- (146) -- (46);
  \draw (23) -- (235) -- (25) -- (235) -- (35);
  \draw (23) -- (236) -- (26) -- (236) -- (36);
  \draw (24) -- (245) -- (25) -- (245) -- (45);
  \draw (24) -- (246) -- (26) -- (246) -- (46);
  \draw (12) -- (1234) -- (13) -- (1234) -- (14) -- (1234) -- (23) -- (1234) -- (24) -- (1234) -- (34);
  \draw (12) -- (1256) -- (15) -- (1256) -- (16) -- (1256) -- (25) -- (1256) -- (26) -- (1256) -- (56);
  \draw (34) -- (3456) -- (35) -- (3456) -- (36) -- (3456) -- (45) -- (3456) -- (46) -- (3456) -- (56);
  \draw (135) -- (min);
  \draw (136) -- (min);
  \draw (145) -- (min);
  \draw (146) -- (min);
  \draw (235) -- (min);
  \draw (236) -- (min);
  \draw (245) -- (min);
  \draw (246) -- (min);
  \draw (1234) -- (min);
  \draw (1256) -- (min);
  \draw (3456) -- (min);
\end{tikzpicture}
\caption{The boundary stratification poset of $\M_3$. The red and blue nodes correspond to strata of type (3)  and 
of type (4), respectively.}\label{fig:poset}
\end{figure}
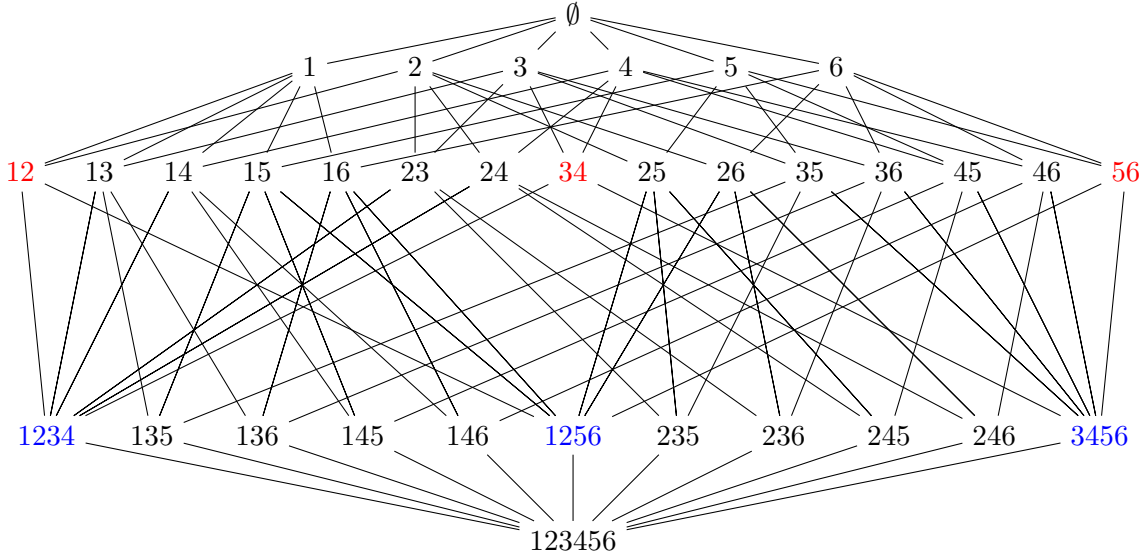

We now turn to the proof of Theorem \ref{thm-main1}. The result will
follow from a sequence of lemmas.  By Proposition \ref{prop-boundary}
there are exactly $2n$ strata of codimension one, each consisting of
tensors where in one slice along a given dimension the subtensor has
rank at most one. In other words, each stratum is described by a
collection of equations of the form (\ref{eq:Gamma}) together with the
inequalities forcing supermodularity. We denote these strata by
$\Gamma_{is}$ where $i =1, \ldots, n$ and $s=1,2$.

We first formulate a lemma that shows that boundary points are mapped
to boundary points under the marginalization $P\mapsto P_B$
(c.f. Lemma \ref{lem:marg}).
\begin{lemma}\label{lem:marg2}
Suppose that $n\geq 4$ and let $B \subset \{1, \ldots, n\}$ with $|B| = m \geq 3$. For $i \in B$,  
if  a point $P$ in $\M_{n}$ lies on $\Gamma_{is}$, then $P_B$ lies in the corresponding stratum $\Gamma_{is}$ 
of the marginal model $\M_{m}$.
\end{lemma}
\begin{proof}
  If $P$ is the image of $(\lambda_1, \lambda_2, A^{(k)} \, : \,
  k=1,\ldots,n)$, by Lemma \ref{lem:marg}, $P_B$ is the image of
  $(\lambda_1, \lambda_2, A^{(k)} \, : \, k \in B)$. Hence if the
  slice $s$ in dimension $i$ of $P$ has rank one, so will the slice
  $s$ in dimension $i$ of $P_B$.
\end{proof}

\begin{proposition} \label{prop-fiber} The preimage of the codimension
  one stratum $\Gamma_{is}$ under $\phi_n$ is
$$\phi_n^{-1}(\Gamma_{is}) \, = \, \{a_{1s}^{(i)}=0\} \cup \{a_{2s}^{(i)}=0\} \cup \Theta_{\lambda_1\lambda_2} \cup \bigcup_{k \neq i} \Theta_{ik}^\circ.$$
\end{proposition}
\begin{proof}
  We first show that $$\{a_{1s}^{(i)}=0\} \cup \{a_{2s}^{(i)}=0\} \cup
  \Theta_{\lambda_1\lambda_2} \cup \bigcup_{k \neq i}
  \Theta_{ik}^\circ \;\;\subset\;\; \phi_n^{-1}(\Gamma_{is}). $$
  Clearly $\{a_{1s}^{(i)}=0\} \cup \{a_{2s}^{(i)}=0\} \cup \Theta_{\lambda_1\lambda_2}$
  lies in the preimage. To show that the preimage contains also
  $\Theta_{ik}^\circ$ for each $k\neq i$, note that the image of a
  point in $\Theta_{ik}^\circ$ is given by
  $$p_{j_1 \cdots j_i \cdots j_k \cdots j_n} = (\lambda_1 a_{1j_i}^{(i)} a_{1j_k}^{(k)}+ \lambda_2 a_{2j_i}^{(i)} a_{2j_k}^{(k)}) \prod_{l \neq i,k} a_{1j_l}^{(l)}. $$
The points on $\Gamma_{is}$ must satisfy $p_{i_1 \cdots s \cdots i_n}
p_{j_1 \cdots s \cdots j_n} = p_{\nu_1 \cdots s \cdots \nu_n} p_{\mu_1
  \cdots s \cdots \mu_n}$ for $\{i_t, j_t \} = \{\nu_t, \mu_t\}$ where
$1 \leq t \neq i \leq n$. The above point satisfies such equations
since
\begin{eqnarray*}
(\lambda_1a_{1s}^{(i)}a_{1i_k}^{(k)}+\lambda_2a_{2s}^{(i)}a_{2i_k}^{(k)})(\lambda_1a_{1s}^{(i)}a_{1j_k}^{(k)}+\lambda_2a_{2s}^{(i)}a_{2j_k}^{(k)})=\\[.2cm]
(\lambda_1a_{1s}^{(i)}a_{1\nu_k}^{(k)}+\lambda_2a_{2s}^{(i)}a_{2\nu_k}^{(k)})(\lambda_1a_{1s}^{(i)}a_{1\mu_k}^{(k)}+\lambda_2a_{2s}^{(i)}a_{2\mu_k}^{(k)}).	
\end{eqnarray*}
Next we show that no other points lie in the preimage. To this end, from
now on suppose that $a_{1s}^{(i)}\neq 0$, $a_{2s}^{(i)}\neq 0$,
$\lambda_1\cdot \lambda_2\neq 0$ and the parameters are not in $\bigcup_{k \neq i}
\Theta_{ik}^\circ$. Hence we can assume that $P\in \Gamma_{is}$ is given by a parameter vector such that for some $j,k \neq i$ the matrices $A^{(j)},A^{(k)}$ have
rank $2$. Consider the marginal
distribution over $\{i,j,k\}$ and denote its coordinates by $q_{u_i
  u_j u_k}$, $u_i,u_j,u_k\in \{1,2\}$.  By Lemma \ref{lem:marg2}, it
is a point in $\M_3$ parameterized by $(\lambda_1,\lambda_2 , A^{(i)},
A^{(j)}, A^{(k)})$, and it satisfies
$q_{s11}q_{s22}=q_{s12}q_{s21}$. A quick computation shows that this
is equivalent to
\begin{equation}\label{eq:params}
	\lambda_1\lambda_2a_{1s}^{(i)}a_{2s}^{(i)}\det(A^{(j)})\det(A^{(k)})\;=\;0.
\end{equation} 
However, by our
assumption, this is impossible.
\end{proof}


Our strategy to prove Theorem \ref{thm-main1} is to intersect
preimages of codimension one strata $\Gamma_{is}$.  
By Proposition \ref{prop-fiber}, this means we must consider intersections
of subsets of the boundary of the parameter space $\Theta$ and
of various subsets of the singular locus of $\phi_n$ in the
interior of $\Theta$. When doing this, we disregard two
types of intersections. The first type consists of subsets of the
parameter space whose points map to the degenerate part of
$\Delta_{2^n-1}$. Since we are interested in nondegenerate points in
the intersections of $\Gamma_{is}$, these kinds of subsets are irrelevant. 
The second type
consists of subsets of the parameter space whose points map to tensors
of rank one. The next proposition justifies the irrelevance of these
subsets.

\begin{proposition} \label{prop-all}The intersection of all
  $\Gamma_{is}$ for $i=1, \ldots, n$ and $s=1,2$ contains all tensors
  of rank one.
\end{proposition}
 \begin{proof} 
   Every tensor in $\Delta_{2^n-1}$ of rank one is the image of a
   parameter vector in $\Theta$ where $\lambda_1 = 0$.  Such a
   parameter vector is in $\Theta_{\lambda_1 \lambda_2}$. By
   Proposition \ref{prop-fiber}, the image of $\Theta_{\lambda_1
     \lambda_2}$ under the parametrization map is contained in every
   $\Gamma_{is}$.
\end{proof}

\smallskip
In Corollary \ref{cor-all} we prove that the intersection of all 
nondegenerate points in $\Gamma_{is}$ for $i=1,\ldots, n$, $s=1,2$ is {\it equal} to the
set of nondegenerate tensors of rank one. 
This intersection gives us the unique $n$-dimensional stratum (type (5)).
Hence, when intersecting preimages of $\Gamma_{is}$ we ignore 
parameters mapping to tensors of rank one
since 
their images
are in every possible intersection. In summary, when we
refer to intersections of $\phi_n^{-1}(\Gamma_{is})$ we consider
only the \emph{relevant} part, meaning only those points that 
do not map to degenerate or rank one tensors.
For instance, by Proposition
\ref{prop-fiber} the relevant part of $\phi_n^{-1}(\Gamma_{is})$
consists of $\{a_{1s}^{(i)}=0\} \cup \{a_{2s}^{(i)}=0\} \cup
\bigcup_{k \neq i} \Theta_{ik}^\circ$.

\begin{lemma} \label{lem-2inter} The relevant part of
  $\phi_n^{-1}(\Gamma_{is}) \cap \phi_n^{-1}(\Gamma_{jt})$ where $i
  \neq j$ is
$$  \{a_{1s}^{(i)}=0\} \cap \{a_{1 t}^{(j)}=0\}  \, \bigcup \, \{a_{2s}^{(i)}=0\} \cap \{a_{2 t}^{(j)}=0\} \, \bigcup \, \Theta_{ij}^\circ .$$
\end{lemma}
\begin{proof} The points in the set $\{a_{1s}^{(i)}=0\} \cap \{a_{2
    t}^{(j)}=0\}$ map to degenerate tensors since in the
  marginalization of the image tensor $P_{\{i,j\}}(X_i =s, X_j = t) =
  0$. A similar argument shows that $\{a_{2s}^{(i)}=0\} \cap
  \{a_{1t}^{(j)}=0\}$ is irrelevant.  So we just need to compute the
  intersection of $\cup_{k \neq i} \Theta_{ik}^\circ$ and
  $\cup_{\bar{k} \neq j} \Theta_{j\bar{k}}^\circ$. When $k = j$ and
  $\bar{k} = i$, we get $\Theta_{ij}^\circ$. Also, $\Theta_{ij}^\circ
  \cap \Theta_{j \bar{k}} = \Theta_j^\circ$ when $\bar{k} \neq i$, and
  $\Theta_{ik}^\circ \cap \Theta_{ji} = \Theta_i^\circ$ when $k \neq
  j$. Both are irrelevant. For the case $k \neq j$ and $\bar{k} \neq
  i$, we either get $\Theta^1$ if $k \neq \bar{k}$, or
  $\Theta_k^\circ$ if $k = \bar{k}$. Again both cases give irrelevant
  subsets.
\end{proof}

\begin{corollary} \label{cor-2inter} The nondegenerate intersection of
  $\Gamma_{is}$ with $\Gamma_{jt}$ where $i \neq j$ is a stratum of
  dimension $2n-1$.  There are ${n \choose 2} 4$ such strata.
\end{corollary}
\begin{proof} The parametrization map $\phi_n$ is generically smooth on $\bigcup_{u =1,2} \{a_{us}^{(i)}=0\} \cap
  \{a_{ut}^{(j)}=0\}$, and a simple parameter count shows that this
  set has dimension equal to $2n-1$.  Together with Lemma \ref{lem-singdim}
this implies the result. 
For each $1 \leq i < j \leq n$ and each choice of $s,t \in \{1,2\}$ we get such a stratum. Hence, there are
${n \choose 2} 4$ of them. 
\end{proof}

\begin{lemma} \label{lem-excp1}
The relevant part of $\phi_n^{-1}(\Gamma_{i1}) \cap \phi_n^{-1}(\Gamma_{i2})$ is 
$$ \{a_{11}^{(i)} =0\} \cap \{a_{22}^{(i)} =0\} \bigcup  \{a_{12}^{(i)} =0\} \cap \{a_{21}^{(i)} =0\} \bigcup  \cup_{k \neq i} \Theta_{ik}^\circ.$$  
\end{lemma}
\begin{proof}
  The intersections $\{a_{11}^{(i)} =0\} \cap \{a_{12}^{(i)} =0\}$ and
  $\{a_{21}^{(i)} =0\} \cap \{a_{22}^{(i)} =0\}$ are empty in the
  parameter space $\Theta$.
\end{proof}

\begin{corollary} \label{cor-excp1}
The nondegenerate intersection $\Gamma_{i1} \cap \Gamma_{i2}$ is a stratum of dimension $2n-1$. There are $n$ such exceptional strata. 
\end{corollary}
\begin{proof} The parametrization map $\phi_n$ is generically
smooth on $\{a_{11}^{(i)} =0\} \cap \{a_{22}^{(i)}
  =0\}$, and on $\{a_{12}^{(i)} =0\} \cap \{a_{21}^{(i)} =0\}$, and
  the dimension of this set is $2n-1$. Together with Lemma \ref{lem-singdim} 
this gives the first statement. The count is obvious.
\end{proof}

\begin{lemma} \label{lem-triple} The relevant part of
  $\phi_n^{-1}(\Gamma_{is}) \cap \phi_n^{-1}(\Gamma_{jt}) \cap
  \phi_n^{-1}(\Gamma_{kv})$ where $i,j,k$ are distinct is
$$  \bigcup_{u =1,2} \{a_{us}^{(i)}=0\} \cap \{a_{u t}^{(j)}=0\} \cap \{a_{uv}^{(k)} = 0 \}.$$
\end{lemma}
\begin{proof} We proceed as in the proof Lemma \ref{lem-2inter}. After
  discarding irrelevant subsets such as $\{a_{1s}^{(i)}=0\} \cap
  \{a_{1 t}^{(j)}=0\} \cap \{a_{2v}^{(k)} = 0 \}$ (since they map to
  degenerate tensors) we also see that the desired intersection
  contains $\Theta_{ij}^\circ \cap \Theta_{ik}^\circ \cap
  \Theta_{jk}^\circ = \Theta^1$. This is also irrelevant.
\end{proof}
This result immediately generalizes to higher-order intersections.
\begin{corollary} \label{cor-many} Let $I \subset \{1, \ldots, n\}$
  where $|I| = \ell \geq 3$. Then for each choice of $s_i \in \{1,2\}$
  for $i \in I$ the nondegenerate intersection $\bigcap_{i \in I}
  \Gamma_{is_i}$ is a stratum of dimension $2n+1 - \ell$. There are
  ${n \choose \ell} 2^\ell$ such strata.
\end{corollary}
\begin{proof} 
  Lemma \ref{lem-triple} implies that the relevant part of $\bigcap_{i
    \in I} \phi_n^{-1}(\Gamma_{is_i})$ is
$$  \bigcup_{u =1,2}  \left( \bigcap_{s_i: i \in I} \{a_{us_i}^{(i)}=0\} \right).$$
Each piece of this union has dimension $2n+1 - \ell$, and since
$\phi_n$ is generically smooth on these sets the
intersection $\bigcap_{i \in I} \Gamma_{is_i}$ is a stratum of the
same dimension. It is easy to count such strata.
\end{proof}

\begin{lemma} \label{lem-excp2} The relevant part of
  $\phi_n^{-1}(\Gamma_{i1}) \cap \phi_n^{-1}(\Gamma_{i2}) \cap
  \phi_n^{-1}(\Gamma_{j1})$ is $\Theta_{ij}^\circ$.  Moreover, this is
  equal to the relevant part of $\phi_n^{-1}(\Gamma_{i1}) \cap
  \phi_n^{-1}(\Gamma_{i2}) \cap \phi_n^{-1}(\Gamma_{j1}) \cap
  \phi_n^{-1}(\Gamma_{j2})$.
\end{lemma} 
\begin{proof}
  We have computed $\phi_n^{-1}(\Gamma_{i1}) \cap
  \phi_n^{-1}(\Gamma_{i2})$ in Lemma \ref{lem-excp1}.  Together with
  Proposition \ref{prop-fiber} we conclude that we need to describe
  the intersection of
$$ \{a_{11}^{(i)} =0\} \cap \{a_{22}^{(i)} =0\} \bigcup  \{a_{12}^{(i)} =0\} \cap \{a_{21}^{(i)} =0\} \bigcup  \cup_{k \neq i} \Theta_{ik}^\circ$$  
with $$ \{a_{11}^{(j)}=0\} \cup \{a_{21}^{(j)}=0\} \bigcup \cup_{k
  \neq j} \Theta_{jk}^\circ.$$ 
Up to symmetry, we get the following
intersections: (i) $\{a_{12}^{(i)}=0,a_{21}^{(i)}=0,a_{11}^{(j)}=0\}$,
(ii) $\Theta^\circ_{ij}$. It is therefore enough to show that the
first set is irrelevant. Let $P$ be a tensor that is in the image of a
point in the set (i). Then in the marginal distribution $P_{\{i, j\}}$
we have $P_{\{i,j\}}(X_i =1, X_j=1) =0$. Hence $P$ is degenerate.
Finally, when we intersect further with $\phi_n^{-1}(\Gamma_{j2}) =
\{a_{12}^{(j)}=0\} \cup \{a_{22}^{(j)}=0\} \cup \bigcup_{k \neq j}
\Theta_{jk}^\circ$, still the only thing that survives as relevant is
$\Theta_{ij}^\circ$.
\end{proof}

\begin{corollary} \label{cor-excp2} For $i\neq j$, the nondegenerate points 
  in the intersection $\Gamma_{i1} \cap \Gamma_{i2} \cap \Gamma_{j1} \cap
  \Gamma_{j2}$ is a stratum of dimension $n+1$.  There are ${ n
    \choose 2}$ such strata.
\end{corollary}
\begin{proof} The nondegenerate intersection given in the statement is
  the image of $\Theta_{ij}^\circ$ by Lemma \ref{lem-excp2}. This
  intersection is not contained in any other $\Gamma_{k1}$ or
  $\Gamma_{k2}$ for $k\neq i,j$ since by Proposition \ref{prop-fiber}
  everything in $\Theta_{ij}^\circ \cap \phi_n^{-1}(\Gamma_{k1})$ maps
  to tensors of rank one. Hence, indeed $\Gamma_{i1} \cap \Gamma_{i2}
  \cap \Gamma_{j1} \cap \Gamma_{j2}$ defines a stratum. Lemma \ref{lem-singdim} implies that the dimension
of this stratum is $n+1$.  Clearly, there are ${n \choose 2}$ such strata.
\end{proof}
\begin{corollary} \label{cor-all} The intersection of all nondegenerate
points in
  $\Gamma_{is}$ for $i=1, \ldots, n$, $s=1,2$ is the unique stratum
  of dimension $n$ consisting of all 
  nondegenerate
  tensors of rank one.
\end{corollary}
\begin{proof} From Proposition \ref{prop-all} the intersection
  contains the set of tensors of rank one. Corollary $\ref{cor-excp2}$
  implies that $\bigcap_{s=1,2} \left( \Gamma_{is} \cap
    \Gamma_{js} \cap \Gamma_{ks} \right)$ is contained in the set of
  tensors of rank one establishing that the intersection of all
  codimension one strata is a stratum. The dimension of the set of
  rank one tensors is $n$.
\end{proof}

\vskip 0.2cm
\noindent Finally, we prove the main theorem. 

\vskip 0.1cm 
\noindent
\emph{Proof of Theorem \ref{thm-main1}}: Proposition
\ref{prop-boundary} implies that the interior of $\M_n$ has dimension
$2n+1$. The above results imply that any parameter vector with
$a_{1s}^{(i)} = 0$ or $a_{2s}^{(i)} =0$ for $s=1,2$ maps to the
algebraic
boundary of $\M_n$. Similarly, any parameter vector in
$\Theta_{ij}^\circ$ for $1 \leq i\neq j \leq n$ as well as a parameter
vector in $\Theta_{\lambda_1\lambda_2}$ is mapped to the boundary of
$\M_n$. The remaining parameter vectors must map to the interior of
$\M_n$, and these points are nonsingular parameter vectors that are in
the interior of $\Theta$.  We will associate the interior of $\M_n$
with the interior of the $n$-dimensional cube $C_n$.

Also by Proposition \ref{prop-boundary},  $\Gamma_{is}$ for
$i=1,\ldots, n$ and $s=1,2$ are precisely the $2n$ boundary strata of
dimension $2n$.  They are in bijection with the $2n$ facets of $C_n$.
By Proposition \ref{prop-fiber}, the preimage of each $\Gamma_{is}$ is $\{a_{1s}^{(i)}=0\} \cup
\{a_{2s}^{(i)}=0\} \cup \Theta_{\lambda_1\lambda_2} \cup \bigcup_{k
  \neq i} \Theta_{ik}^\circ$. 
Lemma \ref{lem-2inter} proves that $\Gamma_{is} \cap
\Gamma_{jt}$ for $i\neq j$ is the image of points in
$\{a_{1s}^{(i)}=0\} \cap \{a_{1 t}^{(j)}=0\} \, \bigcup \,
\{a_{2s}^{(i)}=0\} \cap \{a_{2 t}^{(j)}=0\} \, \bigcup \,
\Theta_{ij}^\circ$.  By Corollary~\ref{cor-2inter} this is the non-exceptional strata of dimension
$k= 2n-1$ and these strata correspond to $(n-2)$-dimensional faces of $C_n$
which are obtained as intersections of nonparallel facets of the cube
(i.e. $i \neq j$).  Lemma~\ref{lem-triple} and Corollary \ref{cor-many} take care of the
non-exceptional strata of dimension $n< k < 2n-1$ as the image of
$\bigcup_{u =1,2} \left( \bigcap_{s_i: i \in I} \{a_{us_i}^{(i)}=0\}
\right)$. This image is the intersection of $\bigcap_{i \in I}
\Gamma_{is_i}$ where $|I| = 2n+1 - k$. They correspond to faces of
$C_n$ of dimension $k-n-1$. This describes all nondegenerate strata
of types (1) and (2) in the statement of the theorem.

The exceptional strata of codimension two ($k = 2n-1$),
that is of type
(3), is 
described by 
Lemma \ref{lem-excp1} and Corollary \ref{cor-excp1}, combined with 
the proof of Lemma 4.5 in \cite{ARSZ}. 
The statement about the exceptional strata 
(type (4))
of dimension $k = n+1$ 
follows from Lemma \ref{lem-excp2} and Corollary \ref{cor-excp2}.  And
finally, the proof of Lemma \ref{lem-excp2} and Corollary
\ref{cor-all} provide the description of the unique $n$-dimensional
stratum given in type (5). $\Box$

\section{Maximum likelihood estimation over $\M_n$} \label{sec:MLE}
In this section we present how our understanding of the boundary
of $\M_n$ provides a partial understanding of the maximum likelihood
estimation over this model class. For $\M_3$, maximum likelihood
estimators are computed exactly.

Suppose an independent sample of size $N>1$ was observed from a binary
distribution. 
We report the data in a tensor of counts $U=(u_{i_1\cdots i_n})$ where
$u_{i_1\cdots i_n}$ is the number of times the event
$\{X_1=i_1,\ldots,X_n=i_n\}$ was observed. The sum of all elements in
$U$ is equal to $N$.  The log-likelihood function
$\ell:\,\Theta\longrightarrow \R$ is
\begin{equation}\label{eq:loglike}
\ell(\theta)\;=\; \sum_{i_1,\ldots,i_n=1}^2 u_{i_1\cdots i_n}\log(p_{i_1\cdots i_n}(\theta)),	
\end{equation}
where $p_{i_1\cdots i_n}(\theta)$ is as in (\ref{eq:paramMn}). In this
section we are interested in maximizing the log-likelihood function
over $\M_n$ to compute a maximum likelihood estimate (MLE) for the
data~$U$. We remark that $\ell(\theta)$ is a strictly concave function
on the entire $\Delta_{2^n-1}$, and if its unique maximizer over the
entire probability simplex is not in $\M_n$, then its maximizer over
$\M_n$, i.e. the MLE, must be on the boundary of $\M_n$.

In our analysis of boundary strata we restricted attention to
nondegenerate tensors in $\Delta_{2^n-1}$. The lemma below ensures that by looking at the data $U$ we can
detect when the MLE is going to lie in this nondegenerate part, and so, when we can apply Theorem \ref{thm-main1}. It implies that if the sample
proportions tensor $Q = \frac{1}{N} U$ lies outside the degenerate
part of $\Delta_{2^n-1}$, then the MLE $\hat P$ over $\M_n$ will also
be nondegenerate.
\begin{lemma}\label{lem:supp}
  Let $\M$ be a model in $\Delta_{k-1}$ for some $k\geq 1$ and let $Q
  = \frac{1}{N} U$ be the sample proportions for data $U$. Then if the
  MLE $\hat P$ for $U$ exists, the support of $Q$ is contained in the
  support of $\hat P$.
\end{lemma}
\begin{proof}
	The MLE is the constrained maximizer over $\M$ of the log-likelihood
	$$
	\sum_{i=1}^k u_i\log P_i\;=\;\sum_{i\in {\rm supp}(Q)}u_i \log P_i.
	$$
	It is equal to $-\infty$ at all points $P$ with $P_i=0$ for some $i \in {\rm supp}(Q)$. 
\end{proof}

\subsection{General results} \label{subsec:general}

In order to solve the optimization problem for the log-likelihood
(\ref{eq:loglike}), one can compute all critical points of
$\ell(\theta)$ over the interior and all boundary strata of
$\M_n$. For many parametrized statistical models the equations
defining these critical points are just rational functions in the
parameter vector $\theta$. This is the case for the latent class
models that we study in this paper.  We will call the number of {\bf
  complex} critical points of $\ell(\theta)$ over a model for generic
data $U$ the \emph{maximum likelihood degree} (ML-degree) of that
model. The ML-degree for general algebraic statistical models were
introduced in~\cite{CHKS} and~\cite{HKS}. In particular, it was shown
that the ML-degree of such a model is a stable number. We will use the
ML-degrees of the boundary strata of $\M_n$ as an indication for the
complexity of solving the maximum likelihood estimation problem. For
instance, if the ML-degree is $\leq 4$, then one can express the MLE
with closed form formulas as a function of $U$.  In particular, if the
ML-degree is equal to one, then the MLE can be expressed as a rational
function of $U$.
 
In order to solve the constrained optimization problem of maximizing the likelihood function, one can employ the following simple scheme:

\begin{itemize}
	\item[(a)] For each stratum $S$ of $\M_n$ list the critical points of the log-likelihood function constrained to its closure $\overline S$.
	\item[(b)] Pick the best point from the list of those critical points that lie in $\M_n$. 
\end{itemize}

Our first observation for this procedure is that we never need to
check all the strata to find a global maximum.  To see this consider
the poset of the boundary stratification as described in the previous
section.  In our search for the global maximum we start from the
maximal element of the poset and move recursively down.  If a global
maximum over the closure $\overline S$ lies in the stratum $S$, there
is no need to optimize over any stratum $S' \preceq S$. As shown
below, for many strata the MLE is guaranteed to lie inside $\M_n$.

A second observation is that maximizing the log-likelihood over most
of the strata is challenging. The defining constraints correspond to
complicated context specific independence
constraints~\cite{boutilier1996context}, and there is as yet no general
theory on how to optimize over these models exactly.  There
are, however, several exceptions including the strata considered in
Section~\ref{sec:stratification}. We begin by introducing
 notation used below: For the data tensor
$U=(u_{i_1\cdots i_n})$ denote by $U^{(s,t)}=(u^{(s,t)}_{i j})$ the
matrix whose $(i,j)$-th entry is the count
of the event $\{X_s=i,X_t=j\}$ and by $U^{(s)}=(u^{(s)}_{i})$ the vector whose entries are the counts of the event $\{X_s=i\}$. \\

\vskip 0.1cm
(a) \emph{Codimension one strata 
(type (2), $k=2n$)}. Each tensor on one of
these $2n$ strata corresponds to a context specific independence
model, such as where $X_2, \ldots, X_n$ are independent conditionally
on $\{X_1 = 1\}$. The ML-degree of the corresponding conditional model
is one; hence, the MLE is expressed as a rational function of the
data:
$$\hat p_{1j_2\cdots j_i\cdots j_n}\;\;=\;\;\frac{u_{1j_2}^{(1 2)}u_{1j_3}^{(1 3)}\cdots u_{1j_n}^{(1 n)}}{N {u_{1}^{(1)}}^{n-2}}.
$$
After normalization, the other slice is a tensor from the model
$\M_{n-1}$. Therefore $\hat p_{2j_2\cdots j_i \cdots j_n}$ can be
computed by employing any procedure that can be used for
$\M_{n-1}$. For instance, in the next section we derive a closed
form
formula for the maximizer on each boundary stratum of $\M_3$. Hence,  in the case of $\M_4$ all 
codimension one strata will also have closed form formulas.  \\

\vskip 0.1cm
(b) \emph{The exceptional codimension two strata of type (3)}. As 
noted in Section~\ref{sec:stratification}, these strata
correspond to simple graphical models over graphs like that in Figure
\ref{fig:star2}.  The ML-degree of this model is one; hence, the MLE
is expressed as a rational function of the data (also
see~\cite[Section~4.4.2]{Lauritzen}):
$$
\hat p_{j_1j_2\cdots j_i\cdots j_n}\;\;=\;\;\frac{u_{j_1j_2}^{(1 2)}u_{j_1j_3}^{(1 3)}\cdots u_{j_1j_n}^{(1 n)}}{N {u_{j_1}^{(1)}}^{n-2}}.
$$
This point is always guaranteed to lie in $\M_n$ and so we never have
to check strata that lie below that in the Hasse diagram defined in
the previous section.  These are the ${n \choose 2}$ strata of type
(4) and the type (5) stratum. Nevertheless, optimizing over these
special strata is simple so we describe it next.

\vskip 0.1cm
(c) \emph{The $(n+1)$-dimensional strata of type (4)}. These
strata correspond to graphical models with one edge and $n-2$
disconnected nodes.  The ML-degree of this model is one. For example,
if $1$ and $2$ are connected by an edge and all other nodes are
disconnected, the MLE is
$$
\hat p_{j_1j_2j_3\cdots j_n}\;\;=\;\;\frac{1}{N^{n-1}}u^{(12)}_{j_1j_2}u^{(3)}_{j_3}\cdots u^{(n)}_{j_n}.
$$

\vskip 0.1cm
(d) \emph{The $n$-dimensional stratum of type (5)}. This
stratum corresponds to the full-independence model and has ML-degree
one.  The MLE over this stratum is simply
$$
\hat p_{j_1j_2\cdots j_i\cdots j_n}\;\;=\;\;\frac{1}{N^n}u_{j_1}^{(1)}u_{j_2}^{(2)}\cdots u_{j_n}^{(n)}.
$$


%
There is one exceptional case, $n=3$, when \emph{all}
strata are defined by binomial equations, in which case the closure of
each stratum corresponds to a log-linear model. The MLE is therefore
uniquely given and can be easily computed.  We discuss this in more
detail in the following subsection.

\subsection{Maximum Likelihood Estimation for $\M_3$ } \label{subsec:MLE-M3}

The binary latent class model for three observed variables in the
probability simplex $\Delta_7$ is parametrized by
$$ p_{ijk} \, = \,  \lambda_1 a_{1i} b_{1j} c_{1k}  + \lambda_2  a_{2i} b_{2j} c_{2k}  $$
where 
$$ A = \left( \begin{array}{cc} a_{11} & a_{12} \\ a_{21} & a_{22} \end{array} \right) \,\,\,\, 
B = \left( \begin{array}{cc} b_{11} & b_{12} \\ b_{21} & b_{22} \end{array} \right) \,\,\,\, 
C = \left( \begin{array}{cc} c_{11} & c_{12} \\ c_{21} & c_{22} \end{array} \right),
$$
are stochastic matrices.  We will depict the resulting tensor $P$
as 
$$ \lambda_1 \left( \begin{array}{cc|cc} a_{11}b_{11}c_{11} & a_{11}b_{11}c_{12}  &  a_{12}b_{11}c_{11} & a_{12}b_{11}c_{12} \\   
a_{11}b_{12}c_{11} & a_{11}b_{12}c_{12}  &  a_{12}b_{12}c_{11} & a_{12}b_{12}c_{12} \end{array} \right) +    
\lambda_2 \left( \begin{array}{cc|cc} a_{21}b_{21}c_{21} & a_{21}b_{21}c_{22}  &  a_{22}b_{21}c_{21} & a_{22}b_{21}c_{22} \\   
a_{21}b_{22}c_{21} & a_{21}b_{22}c_{22}  &  a_{22}b_{22}c_{21} & a_{22}b_{22}c_{22} \end{array} \right).
$$   
\noindent $\M_{3}$ has dimension $7$ with the following stratification given by Theorem
\ref{thm-main1}, c.f. Figure \ref{fig:poset}. \\

\noindent
{\bf 1.} The interior of $\M_{3}$ has dimension $7$. Its Zariski closure is the linear space $\{ p \, : \, \sum p_{ijk} = 1 \}$. 
Its ML-degree is one and the MLE is computed by
$$ {\hat p}_{ijk} = \frac{ u_{ijk} }{u_{+++}}  \quad \quad  i,j,k = 1,2.$$

\noindent {\bf 2.} There are six $6$-dimensional strata. Each is
obtained as the image of those matrices where one entry in the first
row of $A$, $B$, or $C$ is set to $0$, such as $a_{11} = 0$. The
resulting tensor is of the form
$$ \lambda_1 \left( \begin{array}{cc|cc}  0 &   0  &  b_{11}c_{11} & b_{11}c_{12} \\   
0   &  0  &  b_{12}c_{11} & b_{12}c_{12} \end{array} \right) +    
\lambda_2 \left( \begin{array}{cc|cc} a_{21}b_{21}c_{21} & a_{21}b_{21}c_{22}  &  a_{22}b_{21}c_{21} & a_{22}b_{21}c_{22} \\   
a_{21}b_{22}c_{21} & a_{21}b_{22}c_{22}  &  a_{22}b_{22}c_{21} & a_{22}b_{22}c_{22} \end{array} \right).
$$  
Hence its first slice is a rank one matrix whereas its second slice is
generically a rank two matrix. The Zariski closure is defined by
$p_{111}p_{122} - p_{112}p_{121}$ (together with $\sum p_{ijk} -1$)
and forms a log-linear model. From the statistical point of view this
stratum corresponds to the context specific independence model, where
$X_2$ is independent of $X_3$ given $\{X_1=1\}$. Its ML-degree is one
and the MLE is computed by
$$ {\hat p}_{1jk} = \frac{ u_{1j+} \cdot u_{1+k}}{u_{1++} \cdot u_{+++}},   \quad \quad {\hat p}_{2jk} = \frac{ u_{2jk}}{u_{+++}}  \quad \quad j,k = 1,2.$$ 

\vskip 0.2cm
\noindent

\vskip 0.2cm
\noindent
There are 
fifteen
boundary strata of dimension $5$
arising from types (2) and (3). \\

{\bf 3a.} There are twelve strata of the first kind
arising as type (2) strata.
Each is obtained as the image of two types of parameters. The first type of parameters
has one entry in the first 
(or second) 
row of two matrix parameters equal to zero.
The canonical example is $a_{11} = 0$ and $b_{11} = 0$. The resulting
tensor is of the form
$$ \lambda_1 \left( \begin{array}{cc|cc}   0 &   0   &   0  &  0  \\   
0  & 0  &  c_{11} & c_{12} \end{array} \right) +    
\lambda_2 \left( \begin{array}{cc|cc} a_{21}b_{21}c_{21} & a_{21}b_{21}c_{22}  &  a_{22}b_{21}c_{21} & a_{22}b_{21}c_{22} \\   
a_{21}b_{22}c_{21} & a_{21}b_{22}c_{22}  &  a_{22}b_{22}c_{21} & a_{22}b_{22}c_{22} \end{array} \right).
$$  
The second type comes from parameters where one of the matrices has rank one. The corresponding example for the above 
boundary stratum is when $\rank(C) = 1$, in which case, $c_{11} = c_{21} = c$ 
and $c_{12} = c_{22} = \bar{c}$, since $C$ is a stochastic matrix. The resulting tensor is of the form
$$ \lambda_1 \left( \begin{array}{cc|cc} a_{11}b_{11}c & a_{11}b_{11} \bar{c}  &  a_{12}b_{11}c & a_{12}b_{11} \bar{c} \\   
a_{11}b_{12}c & a_{11}b_{12} \bar{c}  &  a_{12}b_{12}c & a_{12}b_{12} \bar{c} \end{array} \right) +    
\lambda_2 \left( \begin{array}{cc|cc} a_{21}b_{21}c & a_{21}b_{21} \bar{c}  &  a_{22}b_{21}c & a_{22}b_{21} \bar{c} \\   
a_{21}b_{22}c & a_{21}b_{22} \bar{c}  &  a_{22}b_{22}c & a_{22}b_{22} \bar{c} \end{array} \right).
$$   

Two (overlapping) slices of both of these tensors are rank one
matrices, namely, the slices corresponding to
$\left( \begin{array}{ccc} p_{111} & p_{112} \\ p_{121} &
    p_{122} \end{array} \right)$ and $\left( \begin{array}{ccc}
    p_{111} & p_{112} \\ p_{211} & p_{212} \end{array} \right)$.  The
Zariski closure is defined by the $2$-minors of
$\left( \begin{array}{ccc} p_{111} & p_{121} & p_{211} \\ p_{112} &
    p_{122} & p_{212} \end{array} \right)$, and it corresponds to two
context specific independence constraints.  Its ML-degree is one and
the MLE is computed by
$$ {\hat p}_{ijk} = \frac{ u_{ij+} \cdot ( u_{++k} - u_{22k})}{ (u_{+++} - u_{221} - u_{222}) \cdot u_{+++} }   \quad ijk \neq 221, 222 \quad \quad
{\hat p}_{22k} = \frac{ u_{22k}}{u_{+++}} \quad k = 1,2 $$
{\bf 3b.} There are three of the second kind (type (3)). Each comes from parameters where
one of the matrices $A,B,C$ is the identity matrix. 
The canonical
example is $a_{12} = 0$ and $a_{21} = 0$. The resulting tensor is of
the form
$$     
\lambda_1 \left( \begin{array}{cc|cc} b_{11}c_{11} & b_{11}c_{12}  &  0  & 0 \\   
b_{12}c_{11} & b_{12}c_{12}  &   0  & 0  \end{array} \right) + \lambda_2 \left( \begin{array}{cc|cc}  0 &  0   &  b_{21}c_{21} & b_{21}c_{22} \\   
0  & 0   &  b_{22}c_{21} & b_{22}c_{22} \end{array} \right).
$$  
Two parallel
slices of these tensors are each rank one matrices, namely, the slices
corresponding to $\left( \begin{array}{ccc} p_{111} & p_{112} \\
    p_{121} & p_{122} \end{array} \right)$ and
$\left( \begin{array}{ccc} p_{211} & p_{212} \\ p_{212} &
    p_{222} \end{array} \right)$.  The Zariski closure is defined by
$p_{111}p_{122} - p_{112}p_{121}$ and $p_{211}p_{222} -
p_{212}p_{221}$, and it corresponds to conditional independence of
$X_2$ and $X_3$ given $X_1$. As indicated in the end of Section~
\ref{sec:stratification}, the ML degree is one and the MLE is computed
by
$$ {\hat p}_{ijk} = \frac{ u_{ij+} \cdot u_{i+k}}{ u_{i++} \cdot u_{+++}} \quad i,j,k = 1,2$$

\vskip 0.2cm
\noindent {\bf 4a.} There are eight $4$-dimensional strata of type (2).
They are defined by the image of matrices where the same entry of the
top row of $A$, $B$, $C$ is zero.
The canonical example is $a_{11}
= b_{11} = c_{11} = 0$. The resulting tensor is of the form
$$ \lambda_1 \left( \begin{array}{cc|cc}  0 &   0  &  0  &  0\\   
0   &  0  &  0  & 1 \end{array} \right) +    
\lambda_2 \left( \begin{array}{cc|cc} a_{21}b_{21}c_{21} & a_{21}b_{21}c_{22}  &  a_{22}b_{21}c_{21} & a_{22}b_{21}c_{22} \\   
a_{21}b_{22}c_{21} & a_{21}b_{22}c_{22}  &  a_{22}b_{22}c_{21} & a_{22}b_{22}c_{22} \end{array} \right).
$$  
The Zariski closure is also a log-linear model whose design matrix $A$ can be chosen to be
$$ A = \left( \begin{array}{rrrrrrrr} 1 & 1 & 0 & 0 & 0 & 0 & -1 & 0 \\
                                                      0 & 0 & 1 & 1 & 0 & 0 & 1  & 0 \\
                                                      0 & 0 & 0 & 0 & 1 & 1 & 1 & 0 \\
                                                      1 & 0 & 1 & 0 & 1  & 0 & 1 & 0 \\
                                                      0 & 1 & 0 & 1 & 0 & 1 &  0 & 0 \\
                                                      0 & 0 & 0 & 0 & 0 & 0 &  0 & 1  \end{array} \right)$$
where the columns correspond to $p_{111}, p_{112}, p_{121}, p_{122}, p_{211}, p_{212}, p_{221}, p_{222}$.
The defining equations are given by the ideal 
$$ I_2 \left( \begin{array}{ccc} p_{111} & p_{121} & p_{211} \\ p_{112} & p_{122} & p_{212} \end{array} \right) + 
     I_2 \left( \begin{array}{cc}   p_{111} & p_{121}  \\ p_{112} & p_{122} \\ p_{211} & p_{221} \end{array} \right)$$
which is minimally generated by five quadrics. The ML-degree is two and the MLE is computed  by choosing one of the two solutions 
obtained as follows.
First ${\hat p}_{222} = \frac{u_{222}}{u_{+++}}$.
Then let 
 \begin{align*}   
\alpha &=  \frac{ u_{111} + u_{112} - u_{221} }{ u_{+++}} &
\beta  &=   \frac{ u_{121} + u_{122} + u_{221} }{ u_{+++}} \\
\gamma  &=   \frac{ u_{211} + u_{212} + u_{221} }{ u_{+++}} &
\delta  &=   \frac{ u_{111} + u_{121} + u_{211} + u_{221} }{ u_{+++}} .
\end{align*}
Then for each root  ${\hat p}_{221}$ of
$$ \delta p_{221}^2  -[(\alpha+\gamma)(\alpha+\beta) + \delta(\gamma+\beta)] p_{221} + \beta\gamma\delta  = 0$$
compute 
\begin{align*}
{\hat p}_{212} &=  \frac{\delta}{\alpha+\gamma} {\hat p}_{221}  + \left( \gamma - \frac{\gamma\delta}{\alpha+\gamma}  \right) \\
{\hat p}_{211} &= -{\hat p}_{212} - {\hat p}_{221} + \gamma \\
{\hat p}_{122} &=  \frac{\delta}{\alpha+\beta} {\hat p}_{221}  + \left( \beta - \frac{\beta\delta}{\alpha+\beta}  \right) \\
{\hat p}_{121} &= -{\hat p}_{122} - {\hat p}_{221} + \beta \\
{\hat p}_{111} &= -{\hat p}_{121} - {\hat p}_{211} - {\hat p}_{221} + \delta \\
{\hat p}_{112} &= -{\hat p}_{111} + {\hat p}_{221} + \alpha.
\end{align*} 
Note that the computations should be done in the exact order given above.  \\

\vskip 0.2cm
\noindent {\bf 4b.} There are three $4$-dimensional strata of type
(4). They are obtained by letting one of the matrices $A$, $B$, $C$ 
have rank one.  A canonical example is where $a_{21}=a_{11}=a$,
$a_{22}=a_{12}= \bar{a}$. The resulting tensor is of the form
$$ \lambda_1 \left( \begin{array}{cc|cc} a b_{11}c_{11} & a b_{11}c_{12}  &  \bar{a} b_{11}c_{11} & \bar{a} b_{11}c_{12} \\   
a b_{12}c_{11} & a b_{12}c_{12}  &  \bar{a} b_{12}c_{11} & \bar{a} b_{12}c_{12} \end{array} \right) +    
\lambda_2 \left( \begin{array}{cc|cc} a b_{21}c_{21} & a b_{21}c_{22}  &  \bar{a} b_{21}c_{21} & \bar{a} b_{21}c_{22} \\   
a b_{22}c_{21} & a b_{22}c_{22}  &  \bar{a} b_{22}c_{21} & \bar{a} b_{22}c_{22} \end{array} \right)
$$   
As indicated in the end of Section~
\ref{sec:stratification}, the ML degree is one and the MLE is computed by 
$$ {\hat p}_{ijk} = \frac{u_{ij+} \cdot u_{++k}}{u_{+++}^2}$$
\vskip 0.2cm

\noindent
{\bf 5.} There is one stratum of dimension three formed by rank one tensors, also known as the independence model on 
three binary random variables. This is a toric model and has ML degree one where the ML estimate is computed by
$$ {\hat p}_{ijk} = \frac{ u_{i++} \cdot u_{+j+} \cdot u_{++k} }{ u_{+++}^3}.$$

\subsection{Simulations}
The exact maximum likelihood estimation for $\M_3$ gives us valuable
insight into the geometry of the likelihood function for the latent
class models. In this subsection
we report on simulations that were designed to unearth this geometry. We also obtain a new perspective into the performance of the EM-algorithm.\\

\vskip 0.1cm
\noindent (a) We say that a point $P \in \Delta_7$ lies in the
attraction basin of a stratum $S$, if, given that the sample
proportions tensor is $P$, the global maximum of the likelihood
function lies in $S$.  In our first simulation we approximate the
relative volumes of the attraction basins of each stratum.  Attraction
basins for strata of type (4) and (5) are lower dimensional and so
have volume zero.

We run $10^6$ iterations, each time sampling $P$ uniformly from
$\Delta_7$ and then sampling data of size $N=1000$ from $P$.  We use
the resulting data tensor to find the MLE. Table \ref{tab:random}
reports our findings.  In $8.38\%$ of cases, the MLE lies in the
interior of $\M_3$.
Quite interestingly, the fifteen 5-dimensional strata attract almost
$50\%$ of the points. In particular, the three special strata of type
(3) attract $17.29\%$ of the points so each of them attracts
approximately $6\%$. This is almost as much as the interior
attracts, and virtually the same as each codimension one
stratum. Since we are trying to estimate the attraction basin
volumes, we omitted the strata of type (4) and (5) from the table. 
In principle, an attraction basin of zero measure could still contain points
that correspond to tables with integer entries, leading to a positive 
probability of the MLE lying on the stratum for data generated as counts.
However, this did not happen in any of our simulations for Table \ref{tab:random}.

\begin{table}[h]
\caption{Relative volume of MLE attraction basins of strata in $\M_3$ for data uniformly distributed over $\Delta_{7}$. 
}\label{tab:random}
\begin{center}
\begin{tabular}{|c|c|c|c|c|}
\hline
 \mbox{$1\times$ $7$-dim} & \mbox{$6\times$ $6$-dim}  & \mbox{$12\times$ $5a$-dim} & \mbox{$3\times$ $5b$-dim} & \mbox{$8\times$ $4a$-dim}  \\ \hline \hline
8.38 & 36.24 & 29.75 & 17.29 & 8.34  \\ \hline
\end{tabular}
\end{center}
\end{table}
 
The fact that codimension two strata attract more points than the
interior and the codimension one strata together may be a bit
counterintuitive at first, but follows directly from the geometry of
the model. The log-likelihood function is a strictly concave function
over $\Delta_7$ with the unique maximum given by the sample
proportions. Its level sets are convex and centered around the sample
proportions $Q=\frac{1}{N} U$. On the other hand, $\M_3$ is highly
concave, as illustrated by its 3-dimensional linear section~in Figure
1 of~\cite{ARSZ}.  It is then natural to expect that lower-dimensional
strata have higher probability of containing the global maximum as
long as the sample proportions
lie outside of $\M_3$. In the next simulation, we argue that this is not a desirable feature of the latent class model.  \\

\vskip 0.1cm
\noindent (b) Suppose that the true data generating distribution lies in $\M_3$ and the corresponding parameters are $\lambda_1=\lambda_2=\frac{1}{2}$ and
$$
A=B=C=\begin{bmatrix}
1-\epsilon & \epsilon\\
\epsilon & 	1-\epsilon
\end{bmatrix}\qquad \mbox{for }\epsilon\in (0,0.5].
$$
If $\epsilon$ is small, all variables are closely correlated with the
unobserved variable. On the other extreme, if $\epsilon=0.5$, all
variables are independent of the unobserved variable.  We generate $N$
samples from the given distribution with a fixed $\epsilon$ and
compute the MLE, repeating this 10,000 times.  We start with $N=1000$
which is a large number for such a small contingency table. In Table
\ref{tab:epsilon} we see that for $\epsilon$ close to $0.5$ the
probability of the MLE landing in the interior of $\M_3$ is small
despite the fact that the data generating distribution lies in the
model and that $N$ is very large. This means that, even when the data
generating distribution lies in the model, with high probability we
can expect estimates to lie on the boundary. This obviously becomes more dramatic
for smaller values of $N=100$ and $N=50$, see Table \ref{tab:epsilon2}
and Table \ref{tab:epsilon3} respectively.  In the last case, even for
small values of $\epsilon$, there is a high probability of hitting the
boundary. This shows that the latent class models must be used with
caution, especially if correlations between variables are small and
the sample size is relatively small.  Finally, we note that in
producing the last row of Table \ref{tab:epsilon3} we observed some
MLEs on the $3$-dimensional strata of rank one tensors. This happens when the data tensor has rank one. Because 
these MLEs are also MLEs over the strata 5b, we report them there.

\begin{table}[h]
\caption{Relative volume of MLE attraction basins of strata in $\M_3$
for the special generating distributions given by $\epsilon$. Sample size $N=1000$, number of iterations $10000$.}\label{tab:epsilon}
\begin{center}
\begin{tabular}{|c|c|c|c|c|c|}
\hline
\mbox{$\epsilon$} & \mbox{$1\times$ $7$-dim} & \mbox{$6\times$ $6$-dim}  & \mbox{$12\times$ $5a$-dim} & \mbox{$3\times$ $5b$-dim} & \mbox{$8\times$ $4a$-dim}  \\ \hline \hline
0.5 & 12.02 & 47.59 & 22.09 & 13.06 & 5.24  \\ \hline
0.4 & 34.52 & 43.87 & 12.13 & 7.94 & 1.54  \\ \hline
0.3 & 99.32 & 0.67 & 0.01 & 0.00 & 0.00  \\ \hline
0.2 & 100 & 0 & 0 & 0 & 0  \\ \hline
0.1 & 100 & 0 & 0 & 0 & 0  \\ \hline
\end{tabular}
\end{center}
\end{table}

\begin{table}[h]
\caption{
Same as in Table \ref{tab:epsilon} but with sample size $N=100$.}\label{tab:epsilon2}
\begin{center}
\begin{tabular}{|c|c|c|c|c|c|}
\hline
\mbox{$\epsilon$} & \mbox{$1\times$ $7$-dim} & \mbox{$6\times$ $6$-dim}  & \mbox{$12\times$ $5a$-dim} & \mbox{$3\times$ $5b$-dim} & \mbox{$8\times$ $4a$-dim}  \\ \hline \hline
0.5 & 10.72 & 45.97 & 22.92 & 14.35 & 6.04  \\ \hline
0.4 & 12.15 & 46.07 & 21.29 & 15.27 & 5.22  \\ \hline
0.3 & 38.00 & 43.62 & 10.84 & 6.36 & 1.18 \\ \hline
0.2 & 80.53 & 17.92 & 1.60 & 0.32 & 0.03 \\ \hline
0.1 & 90.02 & 9.54 & 0.3 & 0.13 & 0.01 \\ \hline
\end{tabular}
\end{center}
\end{table}

\begin{table}[h]
\caption{
Same as in Table \ref{tab:epsilon} but with sample size $N=50$.}\label{tab:epsilon3}
\begin{center}
\begin{tabular}{|c|c|c|c|c|c|}
\hline
\mbox{$\epsilon$} & \mbox{$1\times$ $7$-dim} & \mbox{$6\times$ $6$-dim}  & \mbox{$12\times$ $5a$-dim} & \mbox{$3\times$ $5b$-dim} & \mbox{$8\times$ $4a$-dim} \\ \hline \hline
0.5 & 10.52 & 45.74 & 23.33 & 14.3 & 6.06 \\ \hline
0.4 & 10.83 & 45.24 & 23.36 & 14.67 & 5.90  \\ \hline
0.3 & 21.59 & 47.16 & 17.11 & 11.49 & 2.65 \\ \hline
0.2 & 51.84 & 38.72 & 5.87 & 3.25 & 0.32 \\ \hline
0.1 & 48.59 & 39.37 & 8.33 & 2.42 & 1.29  \\ \hline
\end{tabular}
\end{center}
\end{table}

\vskip 0.1cm
\noindent (c) From the practical point of view it is of interest to study the performance of the EM algorithm, 
for which no realistic global convergence guarantees are known; see \cite{balakrishnan2017} for a more detailed description of the problem. 
In our simulations to this end, we first generate our data in the same scenario 
as above with $\lambda_1=\lambda_2=\frac{1}{2}$, $\epsilon=0.1,\ldots,0.5$, and for sample sizes $N=50,100,1000$. 
We report how many times the EM algorithm was not able to find the global optimum in less than $10$ reruns. 
Given how simple and low-dimensional the model is, we think of $10$ reruns already as a big number. Our main findings are summarized in Figure~\ref{fig:sfig1}. 
When the sample size is large ($N=1000$) this proportion is small if $\epsilon=0.02, 0.05, 0.1,0.2,0.3$. However, for higher $\epsilon$  in more than half cases 
the EM algorithm was not able to find the global optimum. If $N=50$ the results are even more interesting. Note that for $\epsilon=0.4, 0.5$ the situation 
actually looks better than for $N=1000$. This is somewhat counterintuitive at first but easy to explain. High values of $\epsilon$ correspond to ill-behaved 
distributions (close to singularities). If $N$ is very high, the sample distribution lies close, and hence it is also ill-behaved, resulting in a complicated likelihood function.
If the sample size is small, the variance of the sample distribution is much higher, so with relatively high probability the sample distribution will be far 
and better-behaved. In other words, if the correlations between variables are really small, smaller samples may lead to a better-behaved likelihood than big samples. 
For completeness of our discussion we repeat the same computations for a less symmetric set-up where $\lambda_1=\frac{1}{5}$ but the results were very similar and will not be reported here. 

\begin{figure}\label{fig:sfig1}
  \centering
    \includegraphics[width=0.6\linewidth]{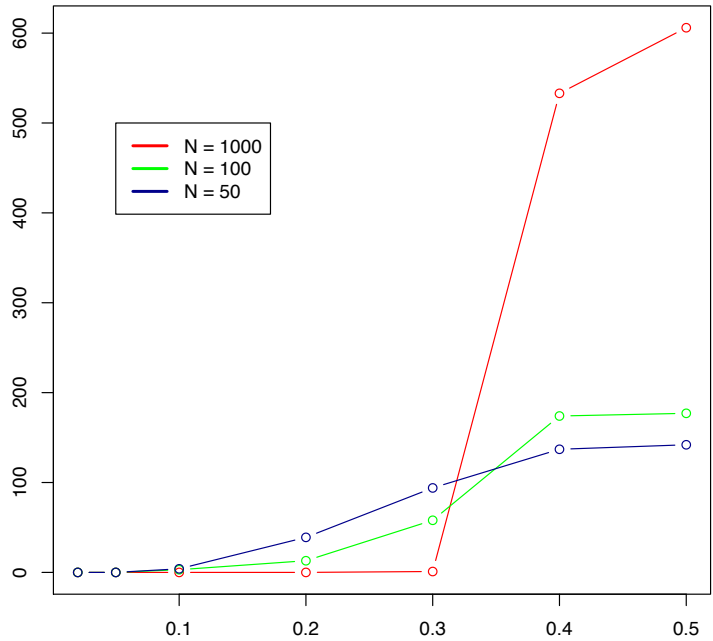}
\caption{The numbers of non-convergers in the EM-algorithm for 1000 experiments and 10 reruns of the 
EM algorithm for each experiment depending on the parameters defining the data-generating distribution. The $x$-axis displays values of $\epsilon$.}
\end{figure}

\subsection{EM attraction basins for $3 \times 3 \times 2$ tensors of $\rank_+ \leq  3$}\label{ss:hector}
We do not have a complete description of the boundary strata for tensors of nonnegative
rank 3, nor formulas for MLEs.  Thus, we present the results of our simulations for $3 \times 3 \times 2$ tensors with $\rank_+ \leq 3$,
denoting the matrix parameters  of appropriate format by $A$, $B$, and $C$.
This model, denoted by $\mathcal M^*$, is a 
full-dimensional, proper subset of $\Delta_{17}$ (i.e. with Zariski closure the full ambient space as in the case of $\mathcal M_3$).
We are interested in giving an estimate for the relative volume of $\mathcal M^*$ in $\Delta_{17}$ and 
in obtaining some preliminary understanding of attraction basins
for distributions sampled from $\Delta_{17}$ under EM. 

We performed two tests. For an arbitrary distribution $P \in \Delta_{17}$,
we ran EM from ten different starting points, recorded the parameters
$\theta_0 \in \Theta$  to which EM converged, and took the optimal estimate.  
Without a full description of the
boundary strata of $\mathcal M^*$, we classified the EM estimate
into four categories:  1) the EM estimate $\theta_0$ contains
strictly positive entries; 2) the EM estimate $\theta_0$ contains
exactly one zero entry in a $3\times 3$ stochastic matrix; 3) the EM estimate $\theta_0$ contains
exactly one zero entry in the $3\times 2$ stochastic matrix; 
4)  the EM estimate $\theta_0$ contains
exactly $k$ zero entries for $k\in\{2,...,11\}$.  The idea is that the numbers of EM estimates per category give approximations to which
points of $\Theta$, either interior or on a boundary face of $\Theta$,
the EM estimates are drawn.   Concretely, these numbers are used to estimate respectively
\begin{enumerate}
	
	\item the relative volume of $\mathcal M^* \subsetneq \Delta_{17}$; 
	
	\item the EM attraction basin proportion for the 6 irreducible components of the algebraic boundary given by a single zero in a $3\times 3$ stochastic matrix
	$A$ or $B$;
	
	\item the EM attraction basin proportion for the 2 irreducible components of the algebraic boundary given by a single zero in the $3\times 2$ stochastic matrix $C$;
	
	\item   the EM attraction basin proportions for intersections of $k$ facets of  $\Theta$.
	
\end{enumerate}

For $10^6$ iterations, the proportions (given as percentages) of these EM attraction
basins are given in Table \ref{tb:emM*1}, where  $1a$-codim corresponds to the relative 
volume of category (2) and $1b$-codim to  the relative volume of category (3).
We separated categories (2) and (3), since
(3) corresponds to a context specific independence model, but (2) does not.

We note that the highest concentration of estimates is in the  faces of $\Theta$ of codimension $4$,
though we have no insight as to why this is the case.  Also the relative volume of $\mathcal M^* \subsetneq \Delta_{17}$,
filling out only approximately $.019\%$ of $\Delta_{17}$ is remarkably small, particularly when compared to relative
volume estimates for $\M_3$ and $\M_{3,3}$.

As a second test, we ran EM for the same $P \in \Delta_{17}$, but with
$10^4$ different starting points.    As expected, EM converged to many local
optima on the nonconvex $\M^*$, with a majority (almost 76\%) in the codim-4 stratum.



\begin{table}[h]
	\caption{Relative volume of EM attraction basins of strata in $\M^*$ using 10 different starting parameters for $10^6$  uniformly distributed points over $\Delta_{17}$.\\ ~}\label{tb:emM*1}
	\begin{center}
		\begin{tabular}{|c|c|c|c|c|c| }
			\hline
			\mbox{$0$-codim} & \mbox{$1a$-codim}  & \mbox{$1b$-codim} & \mbox{$2$-codim}  & \mbox{$3$-codim}& \mbox{$4$-codim}  \\ \hline \hline
			0.019 & 0.2845 &0.0621 &3.4814 &17.0098 & 40.1676    \\ \hline
		\end{tabular}
	\end{center}
	\begin{center}
		\begin{tabular}{|c|c|c|c|c|c|c|c|}
			\hline
			\mbox{$5$-codim} & \mbox{$6$-codim}  & \mbox{$7$-codim} & \mbox{$8$-codim} & \mbox{$9$-codim} & \mbox{$10$-codim}  & \mbox{$11$-codim}  \\ \hline \hline
			25.7120 & 11.2486 & 1.7677 & 0.2249&0.0199&0.0025&0 \\ \hline
		\end{tabular}
	\end{center}
\end{table}

\section{EM fixed point ideals} \label{sec:EM-fixed}

It is well-known that the EM algorithm does not guarantee convergence 
to the global optimum of the likelihood function.
In this section, we study the EM fixed
point ideal introduced in \cite{KRS} that eliminates this
drawback. An EM fixed point for an observed data tensor $U$ is
a parameter vector in $\Theta$ which stays fixed after one iteration
of the EM algorithm. The set of EM fixed points includes the candidates
for the global maxima of the likelihood function; see Lemma
\ref{lemma:EM_fixed_points}. The solution set of the EM fixed point ideal contains all the EM fixed points, in particular, all the global maxima of the
likelihood function. Hence, computing the solution set of the EM fixed ideal allows the computation of all the global maxima for the
likelihood function.  Moreover, for a given model $\M$, the EM fixed
point ideal consists of the equations defining all EM fixed points for
\emph{any} data tensor $U$. Therefore, for a given $\M$, it has to be computed  only once. 
After this computationally intensive task, extracting the MLE for any given data tensor $U$ is relatively easy.

After first describing the equations of the EM fixed point
ideal for $\M_3$ in Proposition \ref{defeqs}, we present the full
prime decomposition of this ideal in Theorem \ref{main}. We
illustrate two uses of this decomposition. First, we
show that using the components of the prime decomposition one can automatically recover the formulas for the
maximum likelihood estimator for various strata that we presented in
Section \ref{subsec:MLE-M3}. 
Second, we point out that the relevant components of
this decomposition that contain entries of stochastic parameter matrices
correspond precisely to the boundary strata of $\M_3$, also reported
in Section \ref{subsec:MLE-M3}. This hints at the usefulness of the EM
fixed point ideal for the discovery of such boundary strata. In fact,
we showcase this discovery process by computing the decomposition of the EM fixed
point ideal of $\M_{3,3}$.  The components we get give the boundary
stratification of $\M_{3,3}$ as reported in \cite{SM17}.

We present a version of the EM algorithm adapted to latent class
models with three observed variables in
Algorithm~\ref{algorithm:EM}.  We no longer assume that the observed
or hidden variables are binary. We let $X_1,X_2,X_3$ be the observed
random variables with $d_1,d_2,d_3$ states, respectively, and we
assume that the hidden variable takes values in $\{1,\ldots,r\}$. We
denote this model by $\M_{d_1 \times d_2 \times d_3,r}$. Our
presentation is based on~\cite[Section 1.3]{pachter2005algebraic}
and~\cite[Algorithm 1]{KRS}.

\begin{algorithm}[H]
\caption{EM algorithm for the latent class model with three observed variables}
\textbf{Input}: Observed data tensor $U \in \mathbb{Z}^{d_1 \times d_2 \times d_3}$. \\
\textbf{Output}: A proposed maximum $\hat{P} \in \Delta_{d_1d_2d_3-1}$ of the log-likelihood function $\ell$ on the model $\mathcal{M}_{d_1 \times d_2 \times d_3,r}$. \\
\vspace{0.2cm}
\textbf{Step 0}: Initialize randomly $(\lambda_1,\ldots,\lambda_r) \in \Delta_{r-1}$, $(a_{i1},\ldots,a_{id_1}) \in \Delta_{d_1-1}$,  $(b_{i1},\ldots,b_{id_2}) \in \Delta_{d_2-1}$, and $(c_{i1},\ldots,c_{id_3}) \in \Delta_{d_3-1}$ for $i=1,\ldots, r$.\\
Run the E-step and M-step until the entries of $P \in \Delta_{d_1d_2d_3-1}$ converge. \\
\textbf{E-Step}: \textit{Estimate the hidden data:}\\
\hspace*{1.5cm} Set $v_{lijk} := \frac{\lambda_l a_{li} b_{lj} c_{lk}}{\sum_{l = 1}^r \lambda_l a_{li} b_{lj} c_{lk}} u_{ijk}$ for $l=1,\ldots,r$, $i = 1, \ldots, d_1$, $j = 1, \ldots, d_2$, and $k = 1,\ldots,d_3$.\\
\textbf{M-Step}: \textit{Maximize the log-likelihood function of the model with complete data using the estimates for the hidden data from the E-step:} \\
\hspace*{1.5cm} Set $\lambda_l := \sum_{i=1}^{d_1} \sum_{i=1}^{d_2} \sum_{i=1}^{d_3} v_{ijkl}/u_{+++}$ for $l = 1, ..., r$. \\
\hspace*{1.5cm} Set $a_{li} := \sum_{j=1}^{d_2} \sum_{k=1}^{d_3} v_{ijkl}/(u_{+++} \lambda_l)$ for $l = 1,\ldots,r$, $i = 1,\ldots,d_1$. \\
\hspace*{1.5cm} Set $b_{lj} := \sum_{i=1}^{d_1} \sum_{k=1}^{d_3} v_{ijkl}/(u_{+++} \lambda_l)$ for $l = 1,\ldots,r$, $j = 1,\ldots,d_2$. \\
\hspace*{1.5cm} Set $c_{lk} := \sum_{i=1}^{d_1} \sum_{j=1}^{d_2} v_{ijkl}/(u_{+++} \lambda_l)$ for $l = 1,\ldots,r$, $k = 1,\ldots,d_3$. \\
\textbf{Update} \textit{the joint distribution for the latent class model:} \\
\hspace*{1.5cm} Set $p_{ijk} := \sum_{l=1}^r \lambda_l a_{li} b_{lj} c_{lk}$ for $i = 1, \ldots, d_1$, $j = 1,\ldots,d_2$, $k=1,\ldots,d_3$.\\
Return $P$.
\label{algorithm:EM}
\end{algorithm}

An \emph{EM fixed point} for an observed data tensor $U$ is an element of  
 $\Theta:=\Delta_{r-1}\times (\Delta_{d_1-1})^r \times (\Delta_{d_2-1})^r \times (\Delta_{d_3-1})^r $ which stays fixed 
 after one E-step and M-step of the EM-algorithm with the input $U$.

\begin{lemma}\label{lemma:EM_fixed_points}
Any $\theta \in \Theta$ to which the EM-algorithm can converge is an EM fixed point.
\end{lemma}

\begin{proof}
Denote the function defined by one step of the EM-algorithm by $\EM(\cdot)$. 
Pick an initial point $\theta^{(0)} \in \Theta$,
and let $\theta^{(k+1)} = \EM(\theta^{(k)})$.  
Assuming that $\theta := \lim_{k \rightarrow \infty} \theta^{(k)}$ exists, then 
$\lim_k \theta^{(k+1)} = \lim_k \EM(\theta^{(k)})$, and since
$\EM$ is continuous, we obtain 
$\theta = \EM(\theta)$.
\end{proof}

Lemma~\ref{lemma:EM_fixed_points} justifies the study of the set of
the EM fixed points as this set contains all possible outputs of the EM
algorithm.  In \cite[Section 3]{KRS}, the set of all EM fixed points
of a latent class model with two observed variables is studied through
the minimal set of polynomial equations that they satisfy. These
equations are called the \emph{EM fixed point equations}.

\begin{proposition} 
\label{defeqs}
The EM fixed point equations for $2 \times 2 \times 2$-tensors of
rank$_+ \leq 2$ on the parameter space $\Theta $ are
\begin{flalign*} a_{\ell i} \left( \sum_{j,k = 1}^2 r_{ijk} b_{\ell j}
    c_{\ell k} \right) &= 0 \hspace{1cm} \mbox{ for all } \;\;\; \ell,
  i = 1, 2, \\ b_{\ell j} \left( \sum_{i,k = 1}^2 r_{ijk} a_{\ell i}
    c_{\ell k} \right) &= 0 \hspace{1cm} \mbox{ for all } \;\;\; \ell,
  j = 1, 2, \\ c_{\ell k} \left( \sum_{i,j = 1}^2 r_{ijk} a_{\ell i}
    b_{\ell j} \right) &= 0 \hspace{1cm} \mbox{ for all } \;\;\; \ell,
  k = 1, 2,
\end{flalign*}
where $[r_{ijk}] = \left[u_{+++} - \frac{u_{ijk}}{p_{ijk}}\right]$.
\end{proposition}

\begin{proof}
The proof is virtually identical to the proof of~\cite[Theorem 3.5]{KRS}. 
\end{proof}


We call the ideal generated by the equations in Proposition \ref{defeqs}
the \textit{EM fixed point ideal} and denote it by $\F$. This ideal is
not prime and it defines a reducible variety. A minimal prime of $\F$
is called \textit{relevant} if it contains none of the $8$ polynomials
$p_{ijk} = \sum_{\ell = 1}^2 a_{\ell i} b_{\ell j} c_{\ell k}$ and none of the six ideals $\left \langle a_{l1},a_{l2} \right \rangle$, $\left \langle b_{l1},b_{l2} \right \rangle$ and $\left \langle c_{l1},c_{l2} \right \rangle$.
Equivalently, an ideal is relevant, if not all of its solutions $P$ has a
coordinate that is identically zero, and after normalizing the parameters, it comes from stochastic matrices.

\begin{theorem}
\label{main}
The radical of the EM fixed point ideal $\F$  for $\M_3$ has precisely $63$ relevant primes consisting of $9$ orbital classes. 
\end{theorem}
\begin{proof}
This proof follows the proof of \cite[Theorem 5.5]{KRS} in using the approach of cellular
components from~\cite{MR1394747}. 
The EM fixed point ideal $\F$ is given by
\[ \left \langle a_{\ell i} \left( \sum_{j,k = 1}^2 r_{ijk} b_{\ell j}
    c_{\ell k} \right), b_{\ell j} \left( \sum_{i,k = 1}^2 r_{ijk}
    a_{\ell i} c_{\ell k} \right), c_{\ell k} \left( \sum_{i,j = 1}^2
    r_{ijk} a_{\ell i} b_{\ell j}\right) : i,j,k,l = 1,2 \right
\rangle. \]  Any prime ideal containing $\F$ contains
either $a_{\ell i}$ or $\sum_{j,k = 1}^2 r_{ijk} b_{\ell j} c_{\ell
  k}$ for $\ell, i \in \{1,2\}$, and either $b_{\ell j}$ or
$\sum_{i,k = 1}^2 r_{ijk} a_{\ell i} c_{\ell k}$ for $\ell, j \in
\{1,2\}$, and either $c_{\ell k}$ or $\sum_{i,j = 1}^2 r_{ijk} a_{\ell
  i} b_{\ell j}$ for $\ell, k \in \{1,2\}$.  We categorize all
primes containing $\F$ according to the set $S$ of parameters $a_{\ell
  i}$, $b_{\ell j}$, and $c_{\ell k}$ contained in them.  The symmetry
group acts on the parameters by permuting the rows of $A$, $B$, and
$C$ simultaneously, the columns of $A$, $B$, and $C$ separately, and
the matrices $A$, $B$, and $C$ themselves. For each orbit that
consists of relevant ideals, we pick one representative $S$ and
compute the \textit{cellular component} $\F_S = ((\F + \langle S
\rangle):(\prod S^c)^\infty)$, where $S^c = \{a_{11},\ldots,a_{22},b_{11},\ldots,b_{22}, c_{11},\ldots, c_{22}\} \setminus S$.  Next
we remove all representatives $S$ such that $\F_T \subset \F_S$ for a
representative $T$ in another orbit.  For each remaining cellular component $\F_S$, we compute its
minimal primes. In each case, we use either the
\texttt{Macaulay2} \texttt{minimalPrimes} function or the linear elimination
sequence from \cite[Proposition 23(b)]{garcia2005algebraic}. Finally, we remove those minimal primes of $\F_S$
that contain a cellular component $\F_T$ for a set $T$ (not necessarily a representative) in
another orbit. The remaining $9$ minimal primes correspond to the rows of
Table \ref{r1} and are uniquely determined by their properties. 
  These
are the set $S$, the degree and codimension, the ranks $rA =
\rank(A)$, $rB = \rank(B)$, and $rC = \rank(C)$ at a generic point. The $63$ ideals are
obtained when counting each orbit with its orbit size in the last
column of Table \ref{r1}.
\end{proof}

\begin{table}[h]
\caption{Minimal primes of EM fixed point ideal $\F$ for $2 \times 2 \times 2$-tensors of rank$_+$ 2.}
\begin{center}
\begin{tabular}{lcccccccccccc} \hline \textbf{Class S} & $\mathbf{|S|}$ & $a$'s & \textbf{$b$'s} & \textbf{$c$'s} & \textbf{deg} & \textbf{codim} & \textbf{rA} & \textbf{rB} & \textbf{rC} & \textbf{orbit} & \textbf{type in Theorem~\ref{thm-main1}}  \\  \hline $\{\emptyset\}$ &0&0&0&0&60&7&1&1&1&1&3-dimensional type 5 \\ &0&0&0&0&48&7&2&2&1&1&4-dimensional type 4 \\ 
&0&0&0&0&48&7&2&1&2&1&4-dimensional type 4 \\
&0&0&0&0&48&7&1&2&2&1&4-dimensional type 4 \\ &0&0&0&0&1&8&2&2&2&1&7-dimensional type 1 \\ $\{a_{11}\}$
 &1&1&0&0&5&8&2&2&2&12&6-dimensional type 2 \\ $\{a_{11}, a_{22}\}$ & 2 & 2 & 0 & 0 & 25 & 8  &2 &2 &2&6&5-dimensional type 3\\ $\{a_{11}, b_{11}\}$ & 2 & 1 & 1 & 0 & 11 & 8 & 2 & 2 & 2&24&5-dimensional type 2 \\ $\{a_{11}, b_{11}, c_{11}\}$ & 3 & 1 & 1 & 1 & 23 & 8 & 2 & 2 & 2 & 16&4-dimensional type 2 \\\hline \hline
\end{tabular} 
\end{center}
\label{r1}
\end{table}

The rows of Table~\ref{r1} correspond to different boundary strata in
Theorem~\ref{thm-main1}, and this correspondence is reported in the last column of the table. 
The orbit sizes in Table~\ref{r1} are twice the number of
corresponding strata in Corollary~\ref{cor:number_of_strata}, except
for the rows represented by $\{\emptyset\}$. 
This is because
the ideal obtained by switching the rows of $A$, $B$ and $C$ is counted as 
distinct from the original ideal, though the tensors in the image of both
sets of parameters are identical with parameters that differ only by label swapping. 
For example, the minimal prime 
represented by $\{a_{21}\}$ is one the 12 ideals in the orbit of the minimal prime represented by $\{a_{11}\}$, 
although $\{a_{11}=0\}$ and $\{a_{21}=0\}$ define the same boundary stratum.



In the next example we explain how to derive the EM fixed points and the potential MLEs from the former type of minimal primes of the EM fixed point ideal.

\begin{example} \label{ex:one_em_minimal_prime}
Consider the minimal prime of the EM fixed point ideal corresponding to $a_{11}=a_{22}=0$:
\begin{align*}
I_1 = &\langle a_{22},a_{11},r_{212}r_{221}-r_{211}r_{222},c_{11}r_{221}+c_{12}r_{222},b_{11}r_{212}+b_{12}r_{222},c_{11
      }r_{211}+c_{12}r_{212},\\
      &b_{11}r_{211}+b_{12}r_{221},r_{112}r_{121}-r_{111}r_{122},c_{21}r_{121}+c_{22}r_{122
      },b_{21}r_{112}+b_{22}r_{122},\\
      &c_{21}r_{111}+c_{22}r_{112},b_{21}r_{111}+b_{22}r_{121} \rangle
   .
\end{align*}
We add to the ideal $I_1$ the ideal of the parametrization map
\begin{align*}
I_2 = & \langle -a_{21}b_{21}c_{21}+p_{111},-a_{21}b_{21}c_{22}+p_{112},-a_{21}b_{22}c_{21}+p_{121},-a_{
      21}b_{22}c_{22}+p_{122},\\
      &-a_{12}b_{11}c_{11}+p_{211},-a_{12}b_{11}c_{12}+p_{212},-a_{12}b_{1
      2}c_{11}+p_{221},-a_{12}b_{12}c_{12}+p_{222} \rangle .
\end{align*}
Eliminating parameters $a_{11},\ldots,c_{22}$ from $I_1+I_2$, gives the ideal
\begin{align*}
J=&\langle p_{212}p_{221}-p_{211}p_{222},r_{221}p_{221}+r_{222}p_{222},r_{211}p_{221}+r_{212
      }p_{222},r_{212}p_{212}+r_{222}p_{222},\\
      &r_{211}p_{212}+r_{221}p_{222},r_{221}p_{211}+r_{
      222}p_{212},r_{212}p_{211}+r_{222}p_{221},r_{211}p_{211}-r_{222}p_{222},\\
      &p_{112}p_{121
      }-p_{111}p_{122},r_{121}p_{121}+r_{122}p_{122},r_{111}p_{121}+r_{112}p_{122},r_{112}p_{
      112}+r_{122}p_{122},\\
      &r_{111}p_{112}+r_{121}p_{122},r_{121}p_{111}+r_{122}p_{112},r_{112
      }p_{111}+r_{122}p_{121},r_{111}p_{111}-r_{122}p_{122},\\
      &r_{212}r_{221}-r_{211}r_{222},r_{
      112}r_{121}-r_{111}r_{122} \rangle
\end{align*}
Finally, we substitute to the ideal $J$ the expressions
\begin{equation*}\label{uijk} 
r_{ijk} = u_{+++} - \frac{u_{ijk}}{p_{ijk}}  
\end{equation*}
and clear the denominators. To obtain an estimate for $p_{111}$, we eliminate all other $p_{ijk}$. This gives the ideal generated by $p_{111}u_{1++}u_{+++}-u_{11+} u_{1+1}$. Hence
$$
p_{111}=\frac{u_{11+} u_{1+1}}{u_{1++}u_{+++}},
$$
as in Section~\ref{sec:MLE}. 
\end{example}

We used the method in Example~\ref{ex:one_em_minimal_prime} to verify
the formulas in Section~\ref{subsec:MLE-M3} for MLEs on different
strata for all cases besides the $3$- and $4$-dimensional strata. For
the $3$- and $4$-dimensional strata, the elimination of $p_{ijk}$'s did
not terminate.

Since the rows of Table~\ref{r1} are in correspondence with the
boundary strata of $\M_3$, we believe that the method of decomposing
the EM fixed point ideal is useful for identifying boundary
strata for models whose geometry is not as well understood as that of
$\M_3$. We illustrate this idea with the decomposition of $\M_{3,3}$.

\begin{theorem}
The radical of the EM fixed point ideal $\F$  for $\M_{3,3}$ has $317$ relevant primes consisting of $21$ orbital classes. 
The properties of these orbital classes are listed in Table~\ref{r3}. 
\end{theorem} 

\begin{longtable}{lcccccccccc}
\caption{Minimal primes of EM fixed point ideal $\F$ for $2 \times 2 \times 2$-tensors of rank$_+$ 3.} \label{r3}\\
\hline 
\multicolumn{1}{l}{\textbf{Set S}} & \multicolumn{1}{c}{$|S|$} & \multicolumn{1}{c}{$a$'s} & \multicolumn{1}{c}{$b$'s} & \multicolumn{1}{c}{$c$'s} & \multicolumn{1}{c}{\textbf{deg}} & \multicolumn{1}{c}{\textbf{cdim}} & \multicolumn{1}{c}{\textbf{rA}} & \multicolumn{1}{c}{\textbf{rB}} & \multicolumn{1}{c}{\textbf{rC}} & \multicolumn{1}{c}{\textbf{orbit}} \\ 
\hline 
\endfirsthead
\multicolumn{10}{c}%
{{\tablename\ \thetable{} \textit{(Continued)}}} \\
\hline 
\multicolumn{1}{l}{\textbf{Set S}} & \multicolumn{1}{c}{$|S|$} & \multicolumn{1}{c}{$a$'s} & \multicolumn{1}{c}{$b$'s} & \multicolumn{1}{c}{\textbf{$c$'s}} & \multicolumn{1}{c}{\textbf{deg}} & \multicolumn{1}{c}{\textbf{cdim}} & \multicolumn{1}{c}{\textbf{rA}} & \multicolumn{1}{c}{\textbf{rB}} & \multicolumn{1}{c}{\textbf{rC}} & \multicolumn{1}{c}{\textbf{orbit}} \\
\hline 
\endhead
\endfoot
\hline \hline
\endlastfoot

$\{\emptyset\}$ & 0 & 0 & 0 & 0 & 121 & 10 & 1 & 1& 1 &1 \\ 
 & 0 & 0 & 0 & 0 & 162 & 9 & 1 & 2 & 2 & 1 \\
  & 0 & 0 & 0 & 0 & 162 & 9 & 2 & 1 & 2 & 1 \\
   & 0 & 0 & 0 & 0 & 162 & 9 & 2 & 2 & 1 & 1 \\
    & 0 & 0 & 0 & 0 & 38 & 10 & 2 & 2 & 2 & $6 \times 1$ \\
     & 0 & 0 & 0 & 0 & 1 & 8 & 2 & 2 & 2 & 1\\
$\{a_{11}\}$ 
&1&1&0&0&10&10&2&2&2&18 \\ 
$\{a_{11}, a_{21}\}$
&2&2&0&0&5&9&2&2&2&18 \\ 
$\{a_{11}, b_{11}\}$ &2&1&1&0&39&10&2&2&2&36 \\ 
$\{a_{11},a_{21},a_{32}\}$ &3&3&0&0&50&11&2&2&2&18 \\
$\{a_{11}, b_{11}, c_{11}\}$ &3&1&1&1&60&11&2&2&2&24 \\
$\{a_{11}, a_{21}, b_{11}, b_{21}\}$ &4&2&2&0&11&10&2&2&2&36 \\ 
$\{a_{11}, a_{22}, b_{11}, b_{22}\}$ &4&2&2&0&8&11&2&2&2&36 \\ 
$\{a_{11}, a_{21}, b_{11}, b_{21}, c_{11}, c_{21}\}$ &6&2&2&2&23&11&2&2&2&24 \\ 
$\{a_{11}, a_{21}, b_{11}, b_{22}, c_{11}, c_{22}\}$ &6&2&2&2&20&12&2&2&2&72 \\ 
$\{a_{11}, a_{22}, b_{11}, b_{22}, c_{11}, c_{22}\}$ &6&2&2&2&23&12&2&2&2&24 
\end{longtable}

Some of the ideals listed in Table~\ref{r3} have further constraints on the $3 \times 2$ stochastic matrices $A$, $B$ and $C$ that cannot be read off directly from the table. These constraints are:
\begin{enumerate}
\item One out of the six ideals of degree $38$ corresponding to $\{\emptyset\}$ contains polynomials $a_{11}a_{22}-a_{12}a_{21}$, $b_{21}b_{32}-b_{22}b_{31}$ and $c_{11}c_{32}-c_{12}c_{31}$. Constraints for the rest of the five ideals are obtained by permuting simultaneously the rows of $A$, $B$ and $C$.
\item The ideal corresponding to $\{a_{11}\}$ contains polynomials $b_{21}b_{32}-b_{22}b_{31}$ and $c_{21}c_{32}-c_{22}c_{31}$.
\item The ideal corresponding to $\{a_{11},b_{11}\}$ contains the polynomial $c_{21}c_{32}-c_{22}c_{31}$.
\item The ideal corresponding to $\{a_{11},a_{21},a_{32}\}$ contains polynomials $b_{11}b_{22}-b_{12}b_{21}$ and $c_{11}c_{22}-c_{12}c_{21}$.
\item The ideal corresponding to $\{a_{11},b_{11},c_{11}\}$ contains polynomials $a_{21}a_{32}-a_{22}a_{31}$, $b_{21}b_{32}-b_{22}b_{31}$ and $c_{21}c_{32}-c_{22}c_{31}$.
\end{enumerate}

The semialgebraic description, boundary stratification and closed formulas for MLEs for $\M_{3,3}$ are obtained in~\cite{SM17}. 
The boundary stratification poset for $\M_{3,3}$ agrees with the one for $\M_{3}$ in Figure~\ref{fig:poset}. 
The parameters that yield different types of boundary strata for $\M_{3,3}$ are included in Table~\ref{r3}:
\begin{enumerate}
\item Interior: $\{\emptyset\}$ ($A,B,C$ rank $2$, no further constraints on $A$, $B$ and $C$).
\item Codimension 1 strata: $\{a_{11}\}^*$, $\{a_{11},a_{21}\}$.
\item Exceptional codimension 2 strata: $\{a_{11},a_{21},a_{32}\}^*$.
\item Codimension 2 strata: $\{a_{11},b_{11}\}^*$, $\{a_{11},a_{21},b_{11},b_{21}\}$.
\item Exceptional codimension 3 strata: $\{\emptyset\}^*$ ($A$, $B$ or $C$ rank $1$). 
\item Codimension 3 strata: $\{ a_{11},b_{11},c_{11}\}^*$, $\{a_{11},a_{21},b_{11},b_{21},c_{11},c_{21}\}$.
\item  Unique codimension 4 stratum: $\{\emptyset\}^*$ ($A$, $B$ and $C$ rank $1$).
\item Other: $\{\emptyset\}^*$ ($A,B,C$ rank $2$, further constraints on $A$, $B$ and $C$), $\{a_{11},a_{22},b_{11},b_{22}\}$, $\{a_{11}, a_{21}, b_{11}, b_{22}, c_{11}, c_{22}\}$, $\{a_{11}, a_{22}, b_{11}, b_{22}, c_{11}, c_{22}\}$.
\end{enumerate}
A star indicates that besides setting the elements in the set to zero, further equation constraints on the parameters 
(either rank constraints from Table~\ref{r3} or other constraints from the list above) are needed to define the stratum. 
Taking these further constraints into account, for a fixed type of boundary stratum, all parametrizations from Table~\ref{r3} are minimal. 
All the rows of Table~\ref{r3} that do not give boundary strata lie on the singular locus of $\M_{3,3}$.

\bibliographystyle{plain}
\bibliography{mle2x2x2}

\begin{thebibliography}{10}

\bibitem{ARM09}
Elizabeth~S. Allman, Catherine Matias, and John~A. Rhodes.
\newblock Identifiability of parameters in latent structure models with many
  observed variables.
\newblock {\em Ann. Statist.}, 37(6A):3099--3132, 2009.

\bibitem{ARSZ}
Elizabeth~S. Allman, John~A. Rhodes, Bernd Sturmfels, and Piotr Zwiernik.
\newblock Tensors of nonnegative rank two.
\newblock {\em Linear Algebra Appl.}, 473:37--53, 2015.

\bibitem{ART2014}
Elizabeth~S. Allman, John~A. Rhodes, and Amelia Taylor.
\newblock A semialgebraic description of the general {M}arkov model on
  phylogenetic trees.
\newblock {\em SIAM J. Discrete Math.}, 28(2):736--755, 2014.

\bibitem{balakrishnan2017}
Sivaraman Balakrishnan, Martin~J. Wainwright, and Bin Yu.
\newblock Statistical guarantees for the {EM} algorithm: From population to
  sample-based analysis.
\newblock {\em Ann. Statist.}, 45(1):77--120, 02 2017.

\bibitem{bishop}
Christopher~M. Bishop.
\newblock {\em Pattern recognition and machine learning}.
\newblock Information Science and Statistics. Springer, New York, 2006.

\bibitem{boutilier1996context}
Craig Boutilier, Nir Friedman, Moises Goldszmidt, and Daphne Koller.
\newblock Context-specific independence in bayesian networks.
\newblock In {\em Proceedings of the Twelfth international conference on
  Uncertainty in artificial intelligence}, pages 115--123. Morgan Kaufmann
  Publishers Inc., 1996.

\bibitem{CHKS}
Fabrizio Catanese, Serkan Ho\c{s}ten, Amit Khetan, and Bernd Sturmfels.
\newblock The maximum likelihood degree.
\newblock {\em Amer. J. Math.}, 128(3):671--697, 2006.

\bibitem{dSL}
Vin de~Silva and Lek-Heng Lim.
\newblock Tensor rank and the ill-posedness of the best low-rank approximation
  problem.
\newblock {\em SIAM J. Matrix Anal. Appl.}, 30(3):1084--1127, 2008.

\bibitem{EM}
A.~P. Dempster, N.~M. Laird, and D.~B. Rubin.
\newblock Maximum likelihood from incomplete data via the {EM} algorithm.
\newblock {\em J. Roy. Statist. Soc. Ser. B}, 39(1):1--38, 1977.
\newblock With discussion.

\bibitem{MR1394747}
David Eisenbud and Bernd Sturmfels.
\newblock Binomial ideals.
\newblock {\em Duke Mathematical Journal}, 84(1):1--45, 1996.

\bibitem{EMdiscussion}
Stephen~E. Fienberg.
\newblock Discussion on the paper by {Dempster}, {Laird}, and {Rubin}.
\newblock {\em J. Roy. Statist. Soc. Ser. B}, 39(1):29--30, 1977.

\bibitem{FienbergLC}
Stephen~E. Fienberg, Patricia Hersh, Alessandro Rinaldo, and Yi~Zhou.
\newblock Maximum likelihood estimation in latent class models for contingency
  table data.
\newblock In {\em Algebraic and geometric methods in statistics}, pages 27--62.
  Cambridge Univ. Press, Cambridge, 2010.

\bibitem{garcia2005algebraic}
Luis~David Garcia, Michael Stillman, and Bernd Sturmfels.
\newblock Algebraic geometry of {B}ayesian networks.
\newblock {\em J. Symbolic Comput.}, 39(3-4):331--355, 2005.

\bibitem{garre2006avoiding}
Francisca~Galindo Garre and Jeroen~K Vermunt.
\newblock Avoiding boundary estimates in latent class analysis by {B}ayesian
  posterior mode estimation.
\newblock {\em Behaviormetrika}, 33(1):43--59, 2006.

\bibitem{geiger2001stratified}
Dan Geiger, David Heckerman, Henry King, and Christopher Meek.
\newblock Stratified exponential families: graphical models and model
  selection.
\newblock {\em Annals of Statistics}, pages 505--529, 2001.

\bibitem{gilula1979singular}
Zvi Gilula.
\newblock Singular value decomposition of probability matrices: Probabilistic
  aspects of latent dichotomous variables.
\newblock {\em Biometrika}, 66(2):339--344, 1979.

\bibitem{Goodman1979}
Leo~A. Goodman.
\newblock On the estimation of parameters in latent structure analysis.
\newblock {\em Psychometrika}, 44(1):123--128, Mar 1979.

\bibitem{haberman}
Shelby~J. Haberman.
\newblock Log-linear models for frequency tables derived by indirect
  observation: maximum likelihood equations.
\newblock {\em Ann. Statist.}, 2:911--924, 1974.

\bibitem{HKS}
Serkan Ho\c{s}ten, Amit Khetan, and Bernd Sturmfels.
\newblock Solving the likelihood equations.
\newblock {\em Found. Comput. Math.}, 5(4):389--407, 2005.

\bibitem{Kruskal76}
Joseph~B. Kruskal.
\newblock More factors than subjects, tests and treatments: an indeterminacy
  theorem for canonical decomposition and individual differences scaling.
\newblock {\em Psychometrika}, 41(3):281--293, 1976.

\bibitem{Kruskal77}
Joseph~B. Kruskal.
\newblock Three-way arrays: rank and uniqueness of trilinear decompositions,
  with application to arithmetic complexity and statistics.
\newblock {\em Linear Algebra and Appl.}, 18(2):95--138, 1977.

\bibitem{KRS}
Kaie Kubjas, Elina Robeva, and Bernd Sturmfels.
\newblock Fixed points {EM} algorithm and nonnegative rank boundaries.
\newblock {\em Ann. Statist.}, 43(1):422--461, 2015.

\bibitem{Lauritzen}
Steffen~L. Lauritzen.
\newblock {\em Graphical models}, volume~17 of {\em Oxford Statistical Science
  Series}.
\newblock The Clarendon Press, Oxford University Press, New York, 1996.
\newblock Oxford Science Publications.

\bibitem{lazarsfeld1950logical}
Paul~F. Lazarsfeld.
\newblock The logical and mathematical foundation of latent structure analysis.
\newblock {\em Studies in Social Psychology in World War II Vol. IV:
  Measurement and Prediction}, pages 362--412, 1950.

\bibitem{lazarsfeld1968lsa}
Paul~F. Lazarsfeld and Neil~W. Henry.
\newblock {\em {Latent Structure Analysis}}.
\newblock Houghton, Mifflin, New York, 1968.

\bibitem{mchugh1956efficient}
Richard~B McHugh.
\newblock Efficient estimation and local identification in latent class
  analysis.
\newblock {\em Psychometrika}, 21(4):331--347, 1956.

\bibitem{MOZ}
Mateusz Micha{\l}ek, Luke Oeding, and Piotr Zwiernik.
\newblock Secant cumulants and toric geometry.
\newblock {\em International Mathematics Research Notices},
  2015(12):4019--4063, 2015.

\bibitem{pachter2005algebraic}
Lior Pachter and Bernd Sturmfels.
\newblock {\em Algebraic statistics for computational biology}, volume~13.
\newblock Cambridge university press, 2005.

\bibitem{SM17}
Anna Seigal and Guido Mont\'ufar.
\newblock Mixtures and products in two graphical models.
\newblock {\em J. Algebraic Statistics}.
\newblock to appear.

\end{thebibliography}
\bigskip
\end{document}